\documentclass{amsart}
\usepackage{color}
\usepackage{amssymb}
\usepackage{amsthm}

\usepackage[unicode=true, pdfusetitle,
 bookmarks=true,bookmarksnumbered=true,bookmarksopen=true,bookmarksopenlevel=1,
 breaklinks=false,pdfborder={0 0 1},backref=false,colorlinks=false]{hyperref}
\hypersetup{pdfstartview=FitH}

\newtheorem{dfn}{Definition }[section]

\newtheorem{thm}[dfn]{Theorem}

\newlength{\fixboxwidth}
\newtheorem{cor}[dfn]{Corollary}
\newtheorem{lem}[dfn]{Lemma}
\theoremstyle{remark}
\newtheorem{rem}[dfn]{Remark}

\newcommand{\mgk}{\mathcal{W}^\alpha_{\alpha_1,\alpha_2}(\Rn)}
\newcommand{\cpu}{c_\infty(1/p,1/u)}
\newcommand{\R}{{\mathbb{R}}}
\newcommand{\Rn}{{\mathbb{R}^n}}
\newcommand{\N}{\mathbb{N}}
\newcommand{\Z}{\mathbb{Z}}
\newcommand{\Zn}{{\mathbb{Z}^n}}
\newcommand{\calP}{\mathcal{P}(\Rn)}
\newcommand{\Plog}{\mathcal{P}^{\log}(\Rn)}
\newcommand{\q}{{q(\cdot)}}
\newcommand{\p}{{p(\cdot)}}
\renewcommand{\u}{{u(\cdot)}}
\newcommand{\SSn}{\mathcal{S}'(\Rn)}
\newcommand{\Sn}{\mathcal{S}(\Rn)}
\newcommand{\Lp}{L_\p(\Rn)}
\newcommand{\Mup}{M_{p(\cdot)}^{u(\cdot)}(\Rn)}
\newcommand{\Muplq}{M_{p(\cdot)}^{u(\cdot)}(\ell_\q)}

\newcommand{\vek}[1]{\boldsymbol{#1}}

\newcommand{\norm}[2]{\left\|\left.{#1}\right|{#2}\right\|}

\newcommand{\Mufwpqx}{\mathcal{E}^{\vek{w},u(\cdot)}_{\p,\q}(\Rn)}

\newcommand{\supp}{\operatorname{supp}}
\newcommand{\mufwpqx}{\mathit{e}^{\vek{w},u(\cdot)}_{\p,\q}}
\newcommand{\mubwpinfx}{\mathit{n}^{\vek{w},u(\cdot)}_{\p,\infty}}
\newcommand{\num}{{\nu,m}}

\begin{document}

\title[Decompositions with atoms and molecules]{Decompositions with atoms and molecules for variable exponent Triebel-Lizorkin-Morrey spaces}


\author[A. Caetano]{Ant\'{o}nio Caetano}
\address{Center for R\&D in Mathematics and Applications, Department of Mathematics, University of Aveiro, 3810-193 Aveiro, Portugal}
\email{acaetano@ua.pt}

\author[H. Kempka]{Henning Kempka}
\address{Department of Fundamental Sciences, PF 100314, University of Applied Sciences Jena, 07703 Jena, Germany}
\email{henning.kempka@eah-jena.de}

\thanks{This research was partially supported by the project ``\emph{Smoothness Morrey spaces with variable exponents}'' approved under the agreement `Projektbezogener Personenaustausch mit Portugal -- A{\c{c}}{\~o}es Integradas Luso-Alem{\~a}s' / DAAD-CRUP. It was also supported through CIDMA (Center for Research and Development in Mathematics and Applications) and FCT (Foundation for Science and Technology) within project UID/MAT/04106/2019.}
\thanks{\copyright 2019. Licensed under the CC-BY-NC-ND 4.0 license http://creativecommons.org/licenses/by-nc-nd/4.0/}

\date{\today}



\begin{abstract}
 We continue the study of the variable exponent Morreyfied Triebel-Lizorkin spaces introduced in a previous paper. Here we give characterizations by means of atoms and molecules. We also show that in some cases the number of zero moments needed for molecules, in order that an infinite linear combination of them (with coefficients in a natural sequence space) converges in the space of tempered distributions, is much smaller than what is usually required. We also establish a Sobolev type theorem for related sequence spaces, which might have independent interest.
\end{abstract}

\maketitle

\section{Introduction}
This paper is a continuation of \cite{CK18}, where we have introduced and presented some properties of the Triebel-Lizorkin-Morrey spaces  ${\mathcal E}^{s(\cdot),\u}_{\p,\q}(\Rn)$, which mix two recent trends in the literature, in this case starting from the Triebel-Lizorkin spaces $F^s_{p,q}(\Rn)$:
\renewcommand{\theenumi}{\Roman{enumi}}%
\begin{enumerate}
	\item on one hand, one \emph{Morreyfies} them in some way, that is, replace the $L_p(\Rn)$ spaces in their construction by Morrey spaces $M^u_p(\Rn)$;
	\item on the other hand, one makes the parameters $s$, $p$, $q$ and $u$ variable.
\end{enumerate}

In \cite{CK18} we have traced a little bit of the history of these trends, which we will not repeat here. We just add that somewhat close to our intentions is the work \cite{YYZ15} and, more recently, \cite{WYYZ18}, where so-called Triebel-Lizorkin-type spaces with variable exponents are studied. In the constant exponents setting --- to which we refer the interested reader to the surveys \cite{Sic12} and \cite{Sic13}, and also to \cite{yuan2010decompositions} and \cite{NNS16} --- these scales include the Triebel-Lizorkin-Morrey spaces. However, in the variable exponents setting that is true only under severe restrictions on the parameters (in particular, \eqref{eq:Horest} below), as we have pointed out in \cite{CK18}.

In \cite{CK18} we have introduced the Triebel-Lizorkin-Morrey spaces  ${\mathcal E}^{s(\cdot),\u}_{\p,\q}(\Rn)$ and proved there an important convolution inequality, which, in particular, allowed us to show that they satisfy a Peetre maximal function characterization, and afterwards concluded the independence of the introduced spaces from the admissible system used.

As we have stressed in \cite{CK18}, one important feature of our approach is that we do not need to make as many restrictions to the parameters as we have seen in approaches by other authors. In particular, there is no need in our approach for the commonly seen restriction

\begin{equation}
\sup_{x \in \Rn}\Big( \frac{1}{p(x)}-\frac{1}{u(x)}\Big) < \frac{1}{\sup p}.
\label{eq:Horest}
\end{equation}

\noindent Moreover, we actually considered 2-microlocal versions $\Mufwpqx$ of the spaces, where the variable smoothness parameter $2^{js(x)}$ is replaced by the more general admissible weights $w_j(x)$.

Continuing from the study made in \cite{CK18}, in the present paper we prove the atomic and molecular characterizations of those spaces given together by Theorems \ref{1st_direction} and \ref{thm:416}. The latter is proved after providing several crucial results, in particular Theorem \ref{convergenceS'}, to which we would also like to draw the attention here. The reason is that it gives a result which is new even if one reduces it to the constant exponent case. It gives sufficient conditions for an infinite linear combination of so-called $[K,L,M]$-molecules to converge in $\SSn$. In connection with constant exponent Triebel-Lizorkin-Morrey spaces ${\mathcal E}^{s,u}_{p,q}(\Rn)$, the usual condition imposed on $L$ (which controls the number of zero moments of the molecules) is
\begin{equation}
L>\sigma_p-s
\label{eq:Lclassical}
\end{equation}
(see, e.g., \cite[Lemma 2.32]{Rosenthal}), where $\sigma_p:=n(1/\min\{1,p\}-1)$. However, reading our Theorem \ref{convergenceS'} for such spaces, a different condition is imposed:
\begin{equation}
L>\frac{n}{u}-s.
\label{eq:Lnew}
\end{equation}
We recover condition \eqref{eq:Lclassical} in our Theorem \ref{convergenceS'B} under the extra assumption $1-p/u<p$, which is really only an extra restriction when $p \leq 1$. So, it is natural to ask which one is weaker: \eqref{eq:Lclassical} or \eqref{eq:Lnew}? As it is easily seen that $1-p/u\leq p$ if and only if $\sigma_p \leq n/u$, we conclude that our new sufficient condition \eqref{eq:Lnew} is weaker if and only if $1-p/u>p$. This holds if and only if the distance between $1/p$ and $1/u$ is greater than 1, since we are also assuming $p\leq u$ here.

On the other hand, we would also like to draw the reader's attention to the fact that the proof of the counterpart of Theorem \ref{convergenceS'} for the variable version of condition \eqref{eq:Lclassical}, established in Theorem \ref{convergenceS'B}, partly relies on a result which might have independent interest, namely a Sobolev type embedding theorem given in Lemma \ref{lem:sobolev} for corresponding sequence spaces.

\section{Preliminaries}
\subsection{General notation}
Here we introduce some of the general notation we use throughout the paper. $\N$, $\N_0$, $\Z$, $\R$ and $\mathbb C$ have the usual meaning as sets of numbers, as well as their $n$-th powers for $n \in \N$. The Euclidean norm in $\Rn$ is denoted by $|\cdot|$, though this notation is also used for the norm of a multi-index, and for Lebesgue measure when it is being applied to (measurable) subsets of $\Rn$. By $\lfloor \cdot \rfloor$ and $\lceil \cdot \rceil$ we mean the usual floor and ceiling functions respectively.

The symbol $\Sn$ stands for the usual Schwartz space of infinitely differentiable rapidly decreasing complex-valued functions on $\Rn$. We take for the (semi)norms generating its locally convex topology the functionals $\mathfrak{p}_N$, for $N\in\N$, defined by
\begin{align*}
\mathfrak{p}_N(\phi):=\sup_{x\in\Rn}(1+|x|)^N\sum_{|\beta|\leq N}|D^\beta\phi(x)|, \quad \phi \in \Sn.
\end{align*}
By { }$\hat{\phi}$ we denote the Fourier transform of $\phi \in \Sn$ in the version
$$ \hat{\phi}(x):=\frac{1}{(2\pi)^{n/2}}\int_{\Rn} e^{-ix\cdot\xi} \phi(\xi)\, d\xi, \quad x \in \Rn,$$
and by ${\phi}^\vee$ we denote the inverse Fourier transform of $\phi$. These transforms are topological isomorphisms in $\Sn$ which extend in the usual way to the space $\SSn$ of tempered distributions, the dual space of $\Sn$, which we endow with the weak topology.

For two complex or extended real-valued measurable functions $f,g$ on $\Rn$ the convolution $f*g$ is given, in the usual way, by
\begin{align*}
(f\ast g)(x):=\int_\Rn f(x-y)g(y)dy, \quad x \in \Rn,
\end{align*}
whenever it makes sense (a.e.).

By $c,c_1,c_\phi,...>0$ we denote constants which may change their value from one line to another. Further, $f\lesssim g$ means that there exists a constant $c>0$ such that $f\leq cg$ holds for a set of variables on which $f$ and $g$ may depend on and which shall be clear from the context. If we write $f\approx g$ then there exists constants $c_1,c_2>0$ with $c_1f\leq g\leq c_2f$. And we shall then say that the expressions $f$ and $g$ are equivalent (across the considered set of variables).

By $Q_{\nu,m}\subset \Rn$, with $\nu \in \Z$ and $m \in \Z^n$, we denote the dyadic closed cube in $\Rn$ which is centered at $2^{-\nu}m$ and has sides parallel to the axes and of length $2^{-\nu}$. Given $d>0$, $dQ_{\nu,m}$ stands for the cube concentric with $Q_{\nu,m}$ and with side length $d2^{-\nu}$. Furthermore, we denote by $B_r(x) \subset \Rn$ the open ball in $\Rn$ with center $x\in\Rn$ and radius $r>0$, and by $Q_r(x) \subset \Rn$ the open cube in $\Rn$ with center $x\in\Rn$ and sides parallel to the axes and of length $2r>0$.\\
The characteristic function $\chi_{\nu.m}$ of a cube $Q_{\nu,m}$ is, as usual, given by
\begin{align*}
\chi_{\nu,m}(x):=\begin{cases}
1,\ &\text{for }x\in Q_{\nu,m}\\
0,&\text{for }x\notin Q_{\nu,m}
\end{cases}.
\end{align*}

The characteristic function $\chi_A$ of any other subset $A$ of $\Rn$ is defined in an analogous way.

Given topological vector spaces $A$ and $B$, the notation $A \hookrightarrow B$ will be used to mean that the space $A$ is continuously embedded into the space $B$.

Finally, the following standard shortcuts are used for $r,s\in(0,\infty]$:
$$\sigma_r := n\left(\frac{1}{\min\{1,r\}}-1\right) \quad \mbox{ and } \quad \sigma_{r,s} := n\left(\frac{1}{\min\{1,r,s\}}-1\right).$$

\subsection{Variable exponent Lebesgue spaces}
The set of variable exponents $\mathcal{P}(\Rn)$ is the collection of all measurable functions $p:\Rn\to(0,\infty]$ with $p^-:=\operatornamewithlimits{ess-inf}_{x\in\Rn}p(x)>0$. Further, we set $p^+:=\operatornamewithlimits{ess-sup}_{x\in\Rn}p(x)$.
For exponents with $p(x)\geq1$ and complex or extended real-valued measurable functions $f$ on $\R^n$ a semi-modular is defined by
\begin{equation*}
\varrho_\p(f):=\int_\Rn \phi_{p(x)}(|f(x)|)\,dx,
\end{equation*}
where
$$
\phi_{p(x)}(t) :=
\begin{cases}
t^{p(x)} & \text{ if } p(x)\in (0,\infty), \\
0 & \text{ if } p(x)=\infty \text{ and } t\in [0,1], \\
\infty & \text{ if } p(x)=\infty \text{ and } t\in(1,\infty], \\
\end{cases}
$$
and the variable exponent Lebesgue space $\Lp$ is given by
$$\Lp:=\{f: \text{ there exists a $\lambda>0$ with }\varrho_\p\left(f/\lambda\right)<\infty\},$$
with their elements being taken in the usual sense of equivalence classes of a.e. coincident functions.
This space is complete and normed, hence a Banach space, with the norm
\begin{align*}
\norm{f}{\Lp}:=\inf\{\lambda>0:\varrho_\p(f/\lambda)\leq 1\}.
\end{align*}
These spaces share many properties with the usual Lebesgue spaces, see for a wide overview \cite{Kovacik}, \cite{DHHR}, \cite{CruzUribe}, but there are also some differences, e.g. they are not translation invariant. By the property 
\begin{equation}\label{eq:Lp/t}
\norm{f}{\Lp}=\norm{|f|^t}{L_{\frac\p t}(\Rn)}^{1/t}\quad\text{for any $t>0$}
\end{equation}
it is also possible to extend the definition of the spaces $\Lp$ to all exponents $p\in\mathcal{P}(\Rn)$. In such more general setting the functional $\norm{\cdot\,}{\Lp}$ need not be a norm, although it is always a quasi-norm.\\
Many theorems for variable Lebesgue spaces $\Lp$ are only valid for exponents $\p$ within a subclass of $\mathcal{P}(\Rn)$ where they satisfy certain regularity conditions. An appropriate subclass in this sense is the set $\Plog$ defined below.
\begin{dfn}
 Let $g:\Rn\to\R$.
 \begin{itemize}
  \item[(i)] We say that $g$ is locally $\log$ H\"older continuous, $g\in C^{\log}_{{\rm loc}}(\Rn)$, if there exists a constant $c_{\log}(g) > 0$ with
  \begin{align*}
   |g(x)-g(y)|\leq\frac{c_{\log}(g)}{\log(e+\frac{1}{|x-y|})}\quad\text{for all }x,y\in\Rn.
  \end{align*}
  \item[(ii)] We say that $g$ is globally $\log$ H\"older continuous, $g\in C^{\log}$, if it is locally $\log$ H\"older continuous and there exist a $g_\infty\in \R$ and a constant $c_\infty(g) > 0$ with
  \begin{align}\label{dfn:cinfty}
   |g(x)-g_\infty|\leq\frac{c_\infty(g)}{\log(e+|x|)}\quad\text{for all }x\in\Rn.
  \end{align}
  \item[(iii)] We write $g\in\Plog$ if $0<g^-\leq g(x)\leq g^+\leq\infty$ with $1/g\in C^{\log}(\Rn)$.
 \end{itemize}

\end{dfn}
Since a control of the quasi-norms of characteristic functions of balls in variable exponent spaces will be crucial for our estimates, we present below a result in that direction and which is an adapted version of \cite[Corollary 4.5.9]{DHHR} to the case $0<p^-\leq p^+\leq\infty$. To obtain it one just has to explore property \eqref{eq:Lp/t} above.
\begin{lem}\label{lem:xinorm}
 Let $p\in\Plog$. Then for all $x_0\in\R^n$ and all $r>0$ we have that 
 \begin{align*}
  \norm{\chi_{B_r(x_0)}}{L_{p(\cdot)}(\Rn)}&\approx \norm{\chi_{Q_r(x_0)}}{L_{p(\cdot)}(\Rn)}\\
  &\approx \begin{cases}
            r^{\frac{n}{p(x)}}\quad&,\;\text{if }r\leq 1\text{ and }x\in B_r(x_0)\\
            r^{\frac{n}{p_\infty}}&,\;\text{if }r\geq1
           \end{cases}.
 \end{align*}
Here we denote $\frac1{p_\infty}:=\Big(\frac1{p}\Big)_\infty$ which is given by \eqref{dfn:cinfty}.
\end{lem}
\section{Variable exponent Morrey spaces}

Now, we can define the Morrey spaces which we are interested in, and which were introduced in \cite{AHS08} (see also the beginning of Section 2.3 in \cite{AC18} for a small survey of literature on variable exponent Morrey spaces).
\begin{dfn}\label{dfn:Morrey}Let $p,u\in\calP$ with $p\leq u$. Then the Morrey space $\Mup$ is the collection of all (complex or extended real-valued) measurable functions $f$ on $\Rn$ with (quasi-norm given by)
\begin{align*}
 \norm{f}{\Mup}:=\sup_{x\in\Rn,r>0}r^{n\left(\frac1{u(x)}-\frac1{p(x)}\right)}\norm{f}{L_{p(\cdot)}(B_r(x))}<\infty.
\end{align*} 
\end{dfn}

For future reference we state and prove the following result giving easy examples of functions belonging to the above Morrey spaces, as long as $p$ satisfies a convenient regularity property.
\begin{lem}\label{lem:Anorm}
Let $u\in\calP$, $p\in\Plog$, $p\leq u$ and $A$ be a measurable and bounded subset of $\Rn$. Then $\chi_A\in\Mup$.
\end{lem}
\begin{proof}
Let $x\in\Rn$ and $r>0$, as in the supremum of the quasi-norm in Definition \ref{dfn:Morrey}. If $0<r\leq1$ we estimate, using Lemma \ref{lem:xinorm},
\begin{align*}
r^{n\left(\frac1{u(x)}-\frac1{p(x)}\right)}\norm{\chi_A}{L_\p(B_r(x))}
&=r^{n\left(\frac1{u(x)}-\frac1{p(x)}\right)}\norm{\chi_A\cdot\chi_{B_r(x)}}{L_\p(\Rn))}\\
&\leq r^{n\left(\frac1{u(x)}-\frac1{p(x)}\right)}\norm{\chi_{B_r(x)}}{L_\p(\Rn))}\\
&\lesssim r^{\frac n{u(x)}}r^{-\frac n{p(x)}}r^{\frac n{p(x)}}\leq r^{\frac n{u(x)}}\leq1.
\end{align*} 
In the case of $r>1$ we use $r^{n/u(x)-n/p(x)}\leq 1$ and Lemma \ref{lem:xinorm} again to obtain
\begin{align*}
r^{n\left(\frac1{u(x)}-\frac1{p(x)}\right)}\norm{\chi_A}{L_\p(B_r(x))}
&\leq \norm{\chi_A\cdot\chi_{B_r(x)}}{L_\p(\Rn))}\\
&\leq \norm{\chi_A}{L_\p(\Rn))}\lesssim R^{\frac n{p_\infty}}<\infty,
\end{align*}
where $R>1$ was chosen such that $A\subset B_R(x)$. Altogether we get $\norm{\chi_A}{\Mup}<\infty$, as required.
\end{proof}
\begin{rem}
As a clear consequence of the above result, $\Mup$ contains also all $L_\infty(\Rn)$-functions which are a.e. equal to zero outside a bounded subset of $\Rn$.
\end{rem}
Next we state the convolution inequality proved in \cite[Thm. 3.3]{CK18}, where $M_{p(\cdot)}^{u(\cdot)}(\ell_{\q})$ stands for the set of all sequences $(f_\nu)_{\nu \in \N_0}$ of (complex or extended real-valued) measurable functions on $\Rn$ such that $\norm{\left(\sum_{\nu=0}^\infty|f_\nu(\cdot)|^{\q}\right)^{1/\q}}{\Mup}$ is finite.
Such a result will be one of the main tools in further results to be presented in this paper. The functions $\eta_{\nu,m}$ considered are given for $\nu\in\N_0$ and $m>0$ by
\begin{align*}
\eta_{\nu,m}(x):=2^{\nu n}(1+2^\nu|x|)^{-m}.
\end{align*}
\begin{thm}[{\cite[Thm. 3.3]{CK18}}]\label{thm:MorreyHardy}
 Let  $p,q\in\Plog$ and $u\in\mathcal{P}(\Rn)$ with $1<p^-\leq p(x)\leq u(x)\leq \sup u < \infty$ and  $q^-,q^+\in(1,\infty)$. For every 
\begin{align*}
m>n+n\max\left(0,\sup_{x\in\Rn}\left(\frac1{p(x)}-\frac1{u(x)}\right)-\frac1{p_\infty}\right)
\end{align*} 
there exists a $c>0$ such that for all $(f_\nu)_\nu \subset M_{p(\cdot)}^{u(\cdot)}(\ell_{\q})$
 \begin{align*}
  \norm{\left(\sum_{\nu=0}^\infty|\eta_{\nu,m}\ast f_\nu(\cdot)|^{q(\cdot)}\right)^{1/\q}}{\Mup}\leq c\norm{\left(\sum_{\nu=0}^\infty|f_\nu(\cdot)|^{\q}\right)^{1/\q}}{\Mup}.
 \end{align*}

\end{thm}

To make our results more accessible we introduce the following abbreviation, which we shall use in the rest of the paper:
\begin{align}\label{eq:cpu}
	\cpu:=\max\left(0,\sup_{x\in\Rn}\left(\frac1{p(x)}-\frac1{u(x)}\right)-\frac1{p_\infty}\right).
\end{align}
It is easily seen that $\cpu=0$ if $\p=\u$ or when $\p=p$ constant.

We shall also need the following lemmas, which we have also already considered in \cite{CK18}. For the meaning of $\norm{\cdot\,}{\Muplq}$, see Definition \ref{def:TLM} below.

\begin{lem}\label{lem:monotony}
	Let $f$ and $g$ be two measurable functions with $0\leq f(x)\leq g(x)$ for a.e. $x\in\Rn$.
	Then it holds
	\begin{align*}
		\norm{f}{\Mup}\leq\norm{g}{\Mup}.
	\end{align*}
\end{lem}

\begin{lem}\label{lem:ttrick}
	Let $p,q,u\in\mathcal{P}(\Rn)$ with $p\leq u$ and $0<t<\infty$. Then for any sequence $(f_\nu)_{\nu\in\N_0}$ of measurable functions it holds
	\begin{align*}
		\norm{(|f_\nu|^t)_\nu}{M_{\frac\p t}^{\frac\u t}(\ell_{\frac \q t})}&=\norm{\left(\sum_{\nu=0}^\infty|f_\nu|^\q\right)^{t/\q}}{M_{\frac\p t}^{\frac\u t}(\Rn)}\\
		&=\norm{(f_\nu)_\nu}{\Muplq}^t,
	\end{align*}
	with the usual modification every time $q(x)=\infty$.
\end{lem}

\begin{lem}\label{lem:Hardy}
	Let $p,q,u\in\mathcal{P}(\Rn)$ with $p\leq u$. For any sequence $(g_j)_{j\in\N_0}$ of non negative measurable functions we denote, for $\delta>0$,
	\begin{align*}
		G_k(x)=\sum_{j=0}^\infty2^{-|k-j|\delta}g_j(x), \quad x \in \Rn, \quad k\in \N_0.
	\end{align*}
	Then it holds
	\begin{align*}
		\norm{(G_k)_k}{\Muplq}\leq c(\delta,q)\norm{(g_j)_j}{\Muplq},
		\intertext{where}
		c(\delta,q)=\max\left(\sum_{j\in\mathbb{Z}}2^{-|j|\delta},\left[\sum_{j\in\mathbb{Z}}2^{-|j|\delta q^-}\right]^{1/q^-}\right).
	\end{align*}
\end{lem}

Finally, we state and prove some results regarding the quasi-norm in the variable exponent Morrey spaces which will prove useful later on.

\begin{lem}
\label{lem:lem3.7}
Let $p,u\in\calP$ with $p\leq u$ and $\inf p >0$. Then $\Mup$ can be equivalently defined as the collection of all (complex or extended real-valued) measurable functions $f$ (on $\Rn$) with
$$\|f\|_{D,p,u} := \sup_{j,x,k} 2^{-jn\left(\frac{1}{u(x)}-\frac{1}{p(x)}\right)}\|f|L_{p(\cdot)}(Q_{j,k})\| < \infty,$$
where the supremum runs over all $j\in\Z$, $x\in \Rn$ and $k\in \Zn$ with $|x-2^{-j}k|_\infty\leq \frac{3}{2}2^{-j}$. Moreover, $\norm{\cdot\,}{\Mup}$ and $\|\cdot\|_{D,p,u}$ are indeed equivalent expressions in $\Mup$.
\end{lem}

\begin{proof}
\underline{First step:} Here we show that $\Mup$  can be equivalently defined by the finiteness of
$$\|f\|_{Q,p,u} := \sup_{x\in\Rn,r>0} r^{n\left(\frac{1}{u(x)}-\frac{1}{p(x)}\right)}\|f|L_\p(Q_r(x))\|.$$
Since $|y|_\infty \leq |y|$, then $B_r(x)\subset Q_r(x)$ and $\|f|\Mup\|\leq \|f\|_{Q,p,u}$. On the other hand, since $|y|\leq \sqrt{n}|y|_\infty$, then $Q_r(x)\subset B_{\sqrt{n}r}(x)$ and, for any $x\in\Rn$ and $r>0$,
\begin{eqnarray*}
r^{n\left(\frac{1}{u(x)}-\frac{1}{p(x)}\right)}\|f|L_\p(Q_r(x))\| & \leq & \sqrt{n}^{n\left(\frac{1}{p(x)}-\frac{1}{u(x)}\right)} (\sqrt{n}r)^{n\left(\frac{1}{u(x)}-\frac{1}{p(x)}\right)} \|f|L_\p(B_{\sqrt{n}r}(x)\| \\
&\leq & \sqrt{n}^{\frac{n}{\inf p}} \|f|\Mup\|.
\end{eqnarray*} 
\underline{Second step:} Here we show that $\|f\|_{D,p,u}\approx\|f\|_{Q,p,u}$, which finishes the proof together with the first step.

Given $j\in\Z$, $x\in \Rn$ and $k\in \Zn$ with $|x-2^{-j}k|_\infty\leq \frac{3}{2}2^{-j}$, define the positive number $r:=2^{-j}$. Notice that $\overset{\circ}{Q}_{j,k} \subset Q_{2r}(x)$, therefore
\begin{eqnarray*}
\lefteqn{2^{-jn\left(\frac{1}{u(x)}-\frac{1}{p(x)}\right)}\|f|L_{p(\cdot)}(Q_{j,k})\|} \\
& \leq & r^{n\left(\frac{1}{u(x)}-\frac{1}{p(x)}\right)}\|f|L_\p(Q_{2r}(x))\| \; \leq \; 2^{\frac{n}{\inf p}} \|f\|_{Q,p,u},
\end{eqnarray*}
hence $\|f\|_{D,p,u} \lesssim \|f\|_{Q,p,u}$.

Now, we prove the opposite estimate.
Given $x\in\Rn$ and $r>0$, let $j\in\Z$ be chosen such that $2^{-j-2}<r\leq 2^{-j-1}$ and pick $k_{r,x}\in\Zn$ such that $|2^jx-k_{r,x}|_\infty\leq\frac{1}{2}$. Clearly,
$$Q_r(x) \subset Q_{2^{-j-1}}(x) \subset \bigcup_k Q_{j,k},$$
where the union runs over all $k\in\Zn$ with $|k-k_{r,x}|_\infty \leq 1$. Notice that the number of cubes in the union above depends only on $n$. We denote by $N$ such a number and have that
\begin{eqnarray*}
\|f\chi_{Q_r(x)}|L_\p(\Rn)\| & \leq & \|f\chi_{\bigcup_k \overset{\circ}{Q}_{j,k}} | L_\p(\Rn) \| \\
& = & \| f \sum_k \chi_{\overset{\circ}{Q}_{j,k}} | L_\p(\Rn)\| \\
& \leq & c_{p^-,n} \sum_k \|f \chi_{Q_{j,k}} | L_\p(\Rn) \|,
\end{eqnarray*}
where $c_{p^-,n}>0$ depends only on $p^-$ and $n$ and the sum runs also over all $k\in\Zn$ with $|k-k_{r,x}|_\infty \leq 1$. Hence
\begin{eqnarray*}
\lefteqn{r^{n\left(\frac{1}{u(x)}-\frac{1}{p(x)}\right)}\|f|L_\p(Q_r(x))\|} \\
& \leq & c_{p^-,n} \sum_k r^{n\left(\frac{1}{u(x)}-\frac{1}{p(x)}\right)}\|f|L_\p(Q_{j,k})\| \\
& \leq & c_{p^-,n} \sum_k 2^{2n\left(\frac{1}{p(x))}-\frac{1}{u(x)}\right)} 2^{-jn\left(\frac{1}{u(x)}-\frac{1}{p(x)}\right)}\|f|L_{p(\cdot)}(Q_{j,k})\| \\
& \leq & c_{p^-,n} 2^\frac{2n}{\inf p} N \|f\|_{D,p,u},
\end{eqnarray*}
using also the fact that, for each considered $k$,
$$|x-2^{-j}k|_\infty \leq |x-2^{-j}k_{r,x}|_\infty + |2^{-j}k_{r,x}-2^{-j}k|_\infty \leq \frac{3}{2}2^{-j}.$$
\end{proof}

\begin{rem} 
\label{rem:remark2}
If $\frac{1}{p}$ is also locally log-H\"older continuous, then for the terms with $j\geq 0$ we can equivalently use $2^{\frac{jn}{p(2^{-j}k)}}$ instead of $2^{\frac{jn}{p(x)}}$ in $\|f\|_{D,p,u}$. This follows by standard arguments from the $\log$ H\"older continuity, see \cite{DHHR}.
\end{rem}

\begin{lem}
\label{lem:lem3.9}
Let $u\in\calP$ and $p\in\Plog$ with $p\leq u$. Let $f$ be of the form 
$$f=\sum_{m\in\Zn} h_{\nu,m} \chi_{\nu,m},$$
for $\nu \in \N_0$ and $h_{\nu,m}\in\mathbb C$. Up to equivalence constants independent of $\nu$, we have that in the calculation of $\|f\|_{D,p,u}$ we only need to consider $j\in\Z$ such that $j\leq \nu$.
\end{lem}

\begin{proof}
Clearly, the new expression, restricting $j\in\Z$ so that $j\leq \nu$, is bounded from above by $\|f\|_{D,p,u}$. On the other hand, we are going to show that for $f$ of the form given, all the terms with $j>\nu$ of the supremum defining $\|f\|_{D,p,u}$ are bounded from above by a suitable constant times a corresponding term for $j=\nu$, which shall conclude the proof.

Given $j\in\Z$ with $j>\nu$, $x\in \Rn$ and $k\in \Zn$ with $|x-2^{-j}k|_\infty\leq \frac{3}{2}2^{-j}$, the corresponding term in the sup defining $\|f\|_{D,p,u}$ is
\begin{equation}
2^{-jn\left(\frac{1}{u(x)}-\frac{1}{p(x)}\right)} \Big\| \sum_{m\in\Zn} h_{\nu,m} \chi_{\nu,m} \chi_{j,k} | L_\p(\Rn) \Big\|.
\label{eq:termj>nu}
\end{equation}
Since necessarily here $j>0$, we can use the preceding remark and instead of $2^{\frac{jn}{p(x)}}$ we use $2^{\frac{jn}{p(2^{-j}k)}}$ in the above expression.

Notice that, by the dyadic structure, either $Q_{j,k}\subset Q_{\nu,m}$ or $\overset{\circ}{Q}_{j,k}\cap Q_{\nu,m}=\emptyset$. The first situation occurs just for one $m\in\Zn$, which we shall denote by $m_{j,k}$. So, \eqref{eq:termj>nu} turns out to be equivalent to
\begin{eqnarray}
\lefteqn{2^{-\frac{jn}{u(x)}} 2^{\frac{jn}{p(2^{-j}k)}} \norm{h_{\nu,m_{j,k}}\chi_{j,k}}{L_\p(\Rn)}} \nonumber \\
& \lesssim & 2^{-\frac{jn}{u(x)}} |h_{\nu,m_{j,k}}| \nonumber \\
& \lesssim & 2^{-\frac{\nu n}{u(x)}} 2^{\frac{\nu n}{p(2^{-\nu}m_{j,k})}} |h_{\nu,m_{j,k}}| 2^{\frac{n}{p^-}} \norm{\chi_{\nu,m_{j,k}}}{L_\p(\Rn)}.
\label{eq:uselemma2.2}
\end{eqnarray}
We have used Lemma \ref{lem:xinorm} to establish the two preceding inequalities. Observe now that $|2^{-\nu}m_{j,k}-2^{-j}k|_\infty \leq 2^{-\nu-1}-2^{-j-1}$, hence
$$|x-2^{-\nu}m_{j,k}|_\infty \leq |x-2^{-j}k|_\infty + |2^{-j}k-2^{-\nu}m_{j,k}|_\infty < \frac{3}{2}2^{-\nu},$$
and using Remark \ref{rem:remark2} we get that \eqref{eq:uselemma2.2} can be estimated by
\begin{eqnarray*}
& \approx & 2^{-\nu n\left(\frac{1}{u(x)}-\frac{1}{p(x)}\right)} \norm{h_{\nu,m_{j,k}} \chi_{\nu,m_{j,k}} \chi_{\nu,m_{j,k}}}{L_\p(\Rn)} \\
& \leq & 2^{-\nu n\left(\frac{1}{u(x)}-\frac{1}{p(x)}\right)} \norm{\sum_{m\in\Zn} h_{\nu,m} \chi_{\nu,m}}{L_\p(Q_{\nu,m_{j,k}})}.
\end{eqnarray*} 
\end{proof}

\section{2-microlocal Triebel-Lizorkin-Morrey spaces and their Peetre maximal function characterization}
We need admissible weight sequences and so called admissible pairs in order to define the spaces under consideration.
\begin{dfn}\label{dfn:admissiblePair}
A pair $(\check{\varphi}, \check{\Phi})$ of functions in $\Sn$ is called admissible if
\begin{align*}
\supp{\varphi}&\subset\{x\in\Rn: \frac12\leq|x|\leq2\}\text{ and }\supp\Phi\subset\{x\in\Rn:|x|\leq2\}\\
\intertext{with}
|\varphi(x)|\geq c&>0\text{ on }\{x\in\Rn:\frac35\leq|x|\leq\frac53\},\\
|\Phi(x)|\geq c&>0\text{ on }\{x\in\Rn:|x|\leq\frac53\}.
\end{align*}
Further, we set $\varphi_j(x):=\varphi(2^{-j}x)$ for $j\geq1$ and $\varphi_0:=\Phi$. Then $(\varphi_j)_{j\in\N_0}\subset\Sn$ and
\begin{align*}
\supp{\varphi_j}\subset\{x\in\Rn:2^{j-1}\leq|x|\leq2^{j+1}\}.
\end{align*}
\end{dfn}
\begin{dfn}\label{weight}
 Let $\alpha_1\leq\alpha_2$ and $\alpha\geq0$ be real numbers. The class of admissible weight sequences $\mgk$ is the collection of all sequences $\vek{w}=(w_j)_{j\in\N_0}$ of measurable functions $w_j$ on $\Rn$ such that
 \begin{itemize}
  \item[(i)] There exists a constant $C>0$ such that 
  \begin{align*}
   0<w_j(x)\leq C w_j(y)(1+2^j|x-y|)^\alpha\text{ for }x,y\in\Rn \text{ and }j\in\N_0; 
  \end{align*}
  \item[(ii)] For all $x\in\Rn$ and $j\in\N_0$ 
  \begin{align*}
   2^{\alpha_1}w_j(x)\leq  w_{j+1}(x)\leq2^{\alpha_2}w_j(x). 
  \end{align*}
 \end{itemize}
\end{dfn}
Now, we can give the definition of 2-microlocal Triebel-Lizorkin-Morrey spaces.
\begin{dfn}\label{def:TLM}
 Let $(\varphi_j)_{j\in\N_0}$ be constructed as in Definition \ref{dfn:admissiblePair} and $\vek{w}\in\mgk$ be admissible weights. Let  $p,q\in\Plog$ and $u\in\mathcal{P}(\Rn)$ with $0<p^-\leq p(x)\leq u(x)\leq \sup u < \infty$ and  $q^-,q^+\in(0,\infty)$. Then
 \begin{align*}
  \Mufwpqx&:=\left\{f\in\SSn:\norm{f}{\Mufwpqx}<\infty\right\}
  \intertext{where}
  \norm{f}{\Mufwpqx}&:=\norm{(w_j(\varphi_j\hat{f})^\vee)_j}{\Muplq}  \\
	&:=\norm{\left(\sum_{j=0}^\infty|w_j(\varphi_j\hat{f})^\vee|^\q\right)^{1/\q}}{\Mup}.
 \end{align*}
\end{dfn}
\begin{rem}
\begin{itemize}
\item[(i)] \label{Eqnormed} First observations about $\Mufwpqx$, e.g. the fact that it is a quasi-normed space, can be seen in \cite{CK18}, to which we refer the reader. \\
\item[(ii)] On the other hand, to show that
\begin{align} \label{eq:doublembed}
\Sn\hookrightarrow\Mufwpqx\hookrightarrow\SSn
\end{align}
we refer to Corollaries \ref{cor:SinE} and \ref{cor:EinSS} below. The proofs come as by-products of the arguments leading to the atomic/molecular characterization of the spaces $\Mufwpqx$. So, it can be considered an interesting application of the tools developed in this paper and constitutes an approach different from the one classically used (e.g. the proof of \cite[(2.3.3/1)]{Tri83}).\\
\item[(iii)] As regards the completeness of $\Mufwpqx$, it follows by standard arguments once the second embedding in \eqref{eq:doublembed} is clear: see Corollary \ref{cor:completeness}.
\end{itemize}
\label{Eelemprop}
\end{rem}
In \cite{CK18} we have also proved the Peetre maximal function characterization of these spaces. Since we shall need to use it below, we recall its statement here.

We need the following notion first.

Given a system $\{\psi_j\}_{j\in\N_0}\subset\Sn$ we set
$\Psi_j=\hat\psi_j\in\Sn$ and define the Peetre maximal
function of $f\in\SSn$ for every $j\in\N_0$ and $a>0$ as
\begin{align*}
    (\Psi_j^*f)_a(x):=\sup_{y\in\R^n}\frac{|(\Psi_j\ast
    f)(y)|}{1+|2^j(y-x)|^a},\quad x\in\R^n.
\end{align*}
We start with two given functions $\psi_0,\psi_1\in\Sn$ and define  $\psi_j(x):=\psi_1(2^{-j+1}x)$, for $x\in\R^n$ and $j\in\N\setminus\{1\}$.
Furthermore, for all $j\in\N_0$ we write, as mentioned, $\Psi_j=\hat{\psi_j}$.

Now, we state the announced Peetre maximal function characterization.
\begin{thm}[{\cite[Thm. 4.5]{CK18}}]\label{thm:lm}
    Let $\vek{w}=(w_j)_{j\in\N_0}\in\mgk$. Assume $p,q\in\Plog$ and $u\in\mathcal{P}(\Rn)$ with $0<p^-\leq p(x)\leq u(x)\leq \sup u < \infty$ and  $q^-,q^+\in(0,\infty)$. Let $R\in\N_0$ with $R>\alpha_2$ and let further $\psi_0,\psi_1$ belong to $\Sn$ with
    \begin{align*}
        D^\beta\psi_1(0)=0,\quad\text{ for
        }0\leq|\beta|< R,
    \end{align*}
    and
    \begin{align*}
        |\psi_0(x)|&>0\quad\text{on}\quad\{x\in\R^n:|x|\leq k\varepsilon\},\\
        |\psi_1(x)|&>0\quad\text{on}\quad\{x\in\R^n:\varepsilon \leq |x| \leq 2k\varepsilon\}
    \end{align*}
    for some $\varepsilon>0$ and $k\in(1,2]$.\\
    For  $a>n\left(\frac{1}{\min(p^-,q^-)}+\cpu\right)+\alpha$ we have that
        \begin{align*}
            \norm{f}{\Mufwpqx}\approx\norm{((\Psi_j\ast
        f)w_j)_j}{\Muplq}\approx\norm{((\Psi_j^*f)_aw_j)_j}{\Muplq}
    \end{align*}
    holds for all $f\in\SSn$.
\end{thm}

Notice that this theorem contains the conclusion that the 2-microlocal Morreyfied spaces of variable exponents given in Definition \ref{def:TLM} are independent of the admissible pair considered.

\section{Atomic and molecular characterizations}
In this section we show a characterization of the spaces $\Mufwpqx$ by atoms and molecules. First of all we need to introduce the corresponding sequence spaces.
\begin{dfn}
Let $p, q, u\in\calP$ with $p\leq u$ and $\vek{w}=(w_k)_{k\in\N_0}\in\mgk$. Then for all complex valued sequences
\begin{align*}
{\lambda}=\{\lambda_{\nu,m}:\nu\in\N_0, m\in\Z^n\}\quad\text{we define }\quad
\mufwpqx:=\left\{\lambda:\norm{\lambda}{\mufwpqx}<\infty \right\}\\
\intertext{where}
\norm{\lambda}{\mufwpqx}:=\norm{\left(\sum_{\nu=0}^\infty\sum_{m\in\Z^n}|w_\nu(2^{-\nu}m)\lambda_{\nu,m}\chi_{\nu,m}(\cdot)|^\q\right)^{1/\q}}{\Mup}
\end{align*}
(with the usual modification every time $q(x)$ equals $\infty$).

\noindent We also define
\begin{align*}
\mubwpinfx&:=\left\{\lambda:\norm{\lambda}{\mubwpinfx}<\infty \right\}\\
\intertext{where}
\norm{\lambda}{\mubwpinfx}&:=\sup_{\nu\in\N_0} \norm{\sum_{m\in\Z^n}w_\nu(2^{-\nu}m)\lambda_{\nu,m}\chi_{\nu,m}(\cdot)}{\Mup}.
\end{align*}
\end{dfn}

\begin{rem}
\begin{itemize}
	\item[(i)] It is easily seen that the above sequence spaces are quasi-normed and that the quasi-norm is a norm when $\min(p^-,q^-)\geq1$.
	\item[(ii)] The following embeddings hold trivially:
	$$\mufwpqx \hookrightarrow \mathit{e}^{\vek{w},u(\cdot)}_{\p,\infty} \hookrightarrow \mubwpinfx.$$
\end{itemize}
\label{smallen}
\end{rem}

Now it is time to define atoms, which are one of the building blocks we consider here.
\begin{dfn}
Let $K,L\in\N_0$ and $d>1$. For each $\nu\in\N_0$ and $m\in\Z^n$ a $C^K$-function $a_{\nu,m}$ is called a $[K,L,d]$-atom (supported near $Q_{\nu,m}$) if
\begin{align*}
\supp a_{\nu,m}&\subset dQ_{\nu,m},\\
\sup_{x\in\Rn}|D^\gamma a_{\nu,m}(x)|&\leq2^{|\gamma|\nu}\quad\text{for }0\leq|\gamma|\leq K
\intertext{and}
\int_\Rn x^\gamma a_{\nu,m}(x)dx&=0\qquad\text{for $\nu\geq1$ and }0\leq|\gamma|< L.
\end{align*}
\end{dfn}
We also give the definition of molecules, where in contrast to atoms the compact support condition is replaced by a decay condition.
\begin{dfn} \label{def:molecules}
Let $K,L\in\N_0$ and $M>0$. For each $\nu\in\N_0$ and $m\in\Z^n$ a $C^K$-function $\mu_{\nu,m}\in C^K(\Rn)$ is called a $[K,L,M]$-molecule (concentrated near $Q_{\nu,m}$) if
\begin{align}
|D^\beta\mu_{\nu,m}(x)|&\leq2^{|\beta|\nu}(1+2^\nu|x-2^{-\nu}m|)^{-M}\quad\text{for }0\leq|\beta|\leq K \notag
\intertext{and}
\int_\Rn x^\beta\mu_{\nu,m}(x)dx&=0\quad\text{for }\nu\geq1\text{ and }0\leq|\beta|< L. \label{moment}
\end{align}
\end{dfn}

For the first direction of the atomic characterization we show that any function in our function spaces can be written as a linear combination of atoms.
\begin{thm} \label{1st_direction}
Let $\vek{w}\in\mgk$, $p, q \in \Plog$ and $u \in\mathcal{P}(\Rn)$ with $0<p^-\leq p(x)\leq u(x)\leq \sup u < \infty$ and  $q^-,q^+\in(0,\infty)$. Further, let $K,L\in\N_0$ and $d>1$.
For each $f\in\Mufwpqx$ there exist $[K,L,d]$-atoms $a_\num \in \Sn$ and $\lambda(f)\in\mufwpqx$ such that 
\begin{align}
f=\sum_{\nu=0}^\infty\sum_{m\in\Z^n}\lambda_\num(f)\, a_\num,\quad\text{convergence in }\SSn, \notag\\
\intertext{and there exists a constant $c>0$ independent of $f$ such that}
\norm{\lambda(f)}{\mufwpqx}\leq c\norm{f}{\Mufwpqx}. \label{norm_estimate}
\end{align}
\end{thm}
\begin{proof}
Since $f \in \SSn$, we have by \cite[Theorem 4.12(a)]{AC16b} that there exist $[K,L,d]$-atoms $a_\num \in \Sn$ and $\lambda(f):=(\lambda_\num)_\num\subset\mathbb C$ such that 
\begin{align} \label{sumf}
f=\sum_{\nu=0}^\infty\sum_{m\in\Z^n}\lambda_\num(f)\, a_\num
\end{align}
with the inner sum taken pointwisely and the outer sum converging in $\SSn$. Afterwards, looking at the proof of \cite[Theorem 4.12]{AC16b} we see in \cite[7.3, Step 2]{AC16b} that the following pointwise estimate holds for a.e. $x \in \Rn$:
\begin{equation}
\left| \sum_{m \in \Z^n} w_\nu(x) \lambda_\num(f) \chi_\num(x) \right| \leq c_a w_\nu(x) (\theta^*_\nu f)_a(x),
\label{eq:pointwise}
\end{equation}
with $a>0$ at our disposal and $c_a>0$ independent of $x \in \Rn$, $\nu \in \N_0$ and $f \in \SSn$. Here $\theta_\nu=2^{\nu n}\theta(2^\nu \cdot)$, $\nu \in \N$, where $\theta_0$, $\theta$ fit, for some $k \in (1,2]$ and $\varepsilon > 0$, in the requirements of Theorem 3.1 of \cite{AC16} for the $\psi_0$, $\psi$ in that theorem and
$$D^\beta\hat{\theta}_1(0)=0\quad\text{ for any } \, \beta \in \N_0^n.$$
Defining now $\psi_0:=\hat{\theta}_0(-\cdot)$ and $\psi_1:=\hat{\theta}(-\frac{\cdot}{2})$, it is easy to see that these $\psi_0$, $\psi_1$ satisfy the conditions of our Theorem \ref{thm:lm} for the same $k$ and $\varepsilon$ as above. We can then apply that theorem with $R>\alpha_2$ and $a > n\left(\frac{1}{\min(p^-,q^-)}+\cpu\right)+\alpha$ and get
$$\norm{f}{\Mufwpqx}\approx\norm{((\theta_\nu^*f)_aw_\nu)_\nu}{\Muplq}.$$
Combining this with \eqref{eq:pointwise} we get using Lemma \ref{lem:monotony}
\begin{equation*}
\norm{\lambda(f)}{\mufwpqx} \approx \norm{\sum_{m \in \Z^n} w_\nu(\cdot) \lambda_\num(f) \chi_\num(\cdot)}{\Muplq} \lesssim \norm{f}{\Mufwpqx}.
\end{equation*}
Using now the hypothesis $f\in\Mufwpqx$ we conclude that $\lambda(f)\in\mufwpqx$ together with the estimate \eqref{norm_estimate}. On the other hand, having now $\lambda(f)\in\mufwpqx$, the proof that the inner sum in \eqref{sumf} also converges in $\SSn$ for the regular distribution given by the pointwise sum can be done by adapting the argument in the second step of the proof of Theorem \ref{convergenceS'} below to our situation here, where we have atoms $a_\num$ instead of molecules $\mu_\num$: on one hand the mentioned argument does not use the assumption on $L$ in that theorem; on the other hand, in the present situation we can choose $M$ fitting the hypotheses of that theorem, since, given any $M>0$, $\left(1+d\,\sqrt{n}/2\right)^{-M}a_\num$ are $[K,L,M]$-molecules concentrated near $Q_\num$ (cf. \cite[Remark 4.3]{AC16b}).
\end{proof}

Before coming to the other direction in the characterizations of $\Mufwpqx$ with atoms and molecules, we clarify the convergence of the sums. Since every $[K,L,d]$-atom supported near $Q_{\nu,m}$ is --- up to a constant factor --- a $[K,L,M]$-molecule concentrated near $Q_{\nu,m}$, it is enough to show the convergence with molecules. Our first theorem in this direction is proved by mixing some ideas from the proofs of \cite[Lemma 3.11]{Kempka10} and \cite[Proposition 4.6]{AC16b}.
\begin{thm} \label{convergenceS'}
Let $\vek{w}\in\mgk$, $p \in \Plog$ and $q, u \in\mathcal{P}(\Rn)$ with $p\leq u$. Let $\lambda\in\mubwpinfx$ and $(\mu_{\nu,m})_{\num}$ be $[K,L,M]$-molecules with 
\begin{align*}
L>-\alpha_1+\frac n{\inf u}\quad\text{and}\quad M>L+2n+2\alpha.
\end{align*}
Then 
\begin{align}\label{eq:fsum}
\sum_{\nu=0}^\infty\sum_{m\in\Z^n}\lambda_{\nu,m}\mu_{\nu,m}\quad\text{converges in }\SSn
\end{align}
and the convergence in $\SSn$ of the inner sum gives the regular distribution obtained by taking the corresponding pointwise convergence. Moreover, the sum 
\begin{align*}
\sum_{(\nu,m)\in\N_0\times\Z^n}\lambda_{\nu,m}\mu_{\nu,m}\quad\text{converges also in }\SSn
\end{align*}
to the same distribution as the iterated sum in \eqref{eq:fsum}.
\end{thm}
\begin{proof}
\underline{First step:} We start by proving an important inequality, which we need in the sequel. We have, for any $\nu\in\N_0$ and $m\in\Z^n$ and with the help of Lemma \ref{lem:xinorm},
\begin{align*}
\norm{\lambda_{\nu,m}w_\nu(2^{-\nu}m)\chi_{\nu,m}}{\Mup} = |\lambda_{\nu,m}|w_\nu(2^{-\nu}m)\norm{\chi_{\nu,m}}{\Mup} &\\
&\hspace{-28em}=|\lambda_{\nu,m}|w_\nu(2^{-\nu}m)\sup_{x\in\Rn,r>0}r^{n\left(\frac1{u(x)}-\frac1{p(x)}\right)}\norm{\chi_\num}{L_\p(B_r(x))}\\
&\hspace{-28em}\geq|\lambda_{\nu,m}|w_\nu(2^{-\nu}m)2^{(-\nu-1) n\left(\frac1{u(2^{-\nu}m)}-\frac1{p(2^{-\nu}m)}\right)}\norm{\chi_{B_{2^{-\nu-1}}(2^{-\nu}m)}}{L_\p(\Rn)}\\
&\hspace{-28em}\approx |\lambda_{\nu,m}|w_\nu(2^{-\nu}m) 2^{(-\nu-1) n\left(\frac1{u(2^{-\nu}m)}-\frac1{p(2^{-\nu}m)}\right)}2^{(-\nu-1) \frac n{p(2^{-\nu}m)}}\\
&\hspace{-28em}\gtrsim |\lambda_{\nu,m}|w_\nu(2^{-\nu}m) 2^{-\nu \frac n{\inf u}},
\end{align*}\\
therefore
\begin{align}
|\lambda_{\nu,m}|w_\nu(2^{-\nu}m) &\lesssim 2^{\nu \frac n{\inf u}} \norm{\lambda_{\nu,m}w_\nu(2^{-\nu}m)\chi_{\nu,m}}{\Mup} \notag \\
&\leq 2^{\nu \frac n{\inf u}} \norm{\lambda}{\mubwpinfx}.\label{eq:lambda}
\end{align}

\smallskip

\underline{Second step:} We show the convergence of the inner sum in \eqref{eq:fsum} both pointwisely a.e. and in $\SSn$ (to the same distribution). Essentially, we only have to repeat the arguments of \cite[7.1, Step 1]{AC16b} with $\nu$ instead of $j$ and using estimate \eqref{eq:lambda} instead of \cite[(7.1) and (2.2)]{AC16b}. So instead of \cite[(7.2)]{AC16b} one gets here for any integer $\kappa > \alpha+n$ using $M>\alpha+n$
\begin{equation}
\int_\Rn \sum_{m\in\Z^n} |\lambda_{\nu,m}\mu_{\nu,m}(x)\phi(x)|\, dx \leq c\, 2^{-\nu (\alpha_1-\frac n{\inf u})} \norm{\lambda}{\mubwpinfx} \mathfrak{p}_{\kappa}(\phi),
\label{eq:4.4*1/2}
\end{equation}
where $c>0$ is independent of $\nu \in \N_0$, $\lambda \in \mubwpinfx$ and $\phi \in \Sn$. In particular, we used the following easy consequence of the properties of the weight sequence
\begin{align}\label{eq:weightestimate}
1\lesssim 2^{-\nu\alpha_1}w_\nu(2^{-\nu}m)(1+|x|)^\alpha(1+2^\nu|x-2^{-\nu}m|)^\alpha,
\end{align}
with the involved constant independent of $x \in \Rn$, $\nu \in \N_0$ and $m \in \Z^n$.

As in \cite[7.1, Step 1]{AC16b}, this shows that the inner sum in \eqref{eq:fsum} is (absolutely) convergent a.e.. The rest follows as in \cite[7.1, Step 1]{AC16b}.

\smallskip

\underline{Third step:} Here we show the convergence in $\SSn$ of the outer  sum in \eqref{eq:fsum}. This follows if we show that there exists $N \in \N$ and $c>0$ such that
\begin{equation}
\sum_{\nu=0}^\infty \left| \int_\Rn \sum_{m\in\Z^n} \lambda_{\nu,m} \mu_{\nu,m}(x)\phi(x)\, dx \right| \leq c\, \mathfrak{p}_N(\phi)
\label{eq:4.4*}
\end{equation}
for all $\phi \in \Sn$.

Consider $\nu \in \N$. By the convergence of the inner sum in $\SSn$ and by the moment conditions \eqref{moment} of $\mu_{\nu,m}$,
\begin{align*}
\left|\int_\Rn\sum_{m\in\Z^n}\lambda_\num\mu_\num(x)\phi(x)dx\right|&\\
&\hspace{-10em}\leq \sum_{m\in\Z^n}|\lambda_\num| \left|\int_\Rn \mu_\num(x)\Big(\phi(x)-\sum_{|\beta|<L}\frac{D^\beta\phi(2^{-\nu}m)}{\beta!}(x-2^{-\nu}m)^\beta\, \Big)dx  \right|\\
&\hspace{-10em}\leq \sum_{m\in\Z^n} |\lambda_\num| \int_\Rn  (1+2^\nu |x-2^{-\nu}m|)^{-M} \sum_{|\beta|=L}\frac{|D^\beta\phi(\xi)|}{\beta!}|x-2^{-\nu}m|^Ldx ,
\end{align*}
where in the last line we have used the controlled decay of $\mu_{\nu,m}$ (cf. Definition \ref{def:molecules}) and Taylor's formula, where $\xi$ lays on the line segment joining $x$ and $2^{-\nu}m$. Now we proceed by using \eqref{eq:weightestimate}, the easy estimate
\begin{equation}
(1+|\xi|)^\kappa (1+2^\nu |x-2^{-\nu}m|)^\kappa \geq (1+|x|)^\kappa
\label{eq:easyestimate}
\end{equation}
for some $\kappa>0$ at our disposal and \eqref{eq:lambda}:
\begin{align}
\left|\int_\Rn\sum_{m\in\Z^n}\lambda_\num\mu_\num(x)\phi(x)dx\right|& \notag \\
&\hspace{-10em} \lesssim \sum_{m\in\Z^n} |\lambda_\num| \int_\Rn (1+2^\nu |x-2^{-\nu}m|)^{-M} 2^{-\nu L} (1+2^\nu |x-2^{-\nu}m|)^{L} \mathfrak{p}_{\max\{\lceil \kappa \rceil,L\}}(\phi) \notag \\ 
&\hspace{-5em} \times (1+|\xi|)^{-\kappa} 2^{-\nu \alpha_1} w_\nu(2^{-\nu}m) (1+|x|)^\alpha (1+2^\nu |x-2^{-\nu}m|)^{\alpha}\, dx \notag \\ 
&\hspace{-10em}\lesssim2^{-\nu (L+\alpha_1-\frac n{\inf u})} \norm{\lambda}{\mubwpinfx} \mathfrak{p}_{\max\{\lceil \kappa \rceil,L\}}(\phi) \notag \\ 
&\hspace{-5em} \times \int_\Rn \sum_{m\in\Z^n}(1+2^\nu|x-2^{-\nu}m|)^{-M+L+\kappa+\alpha}(1+|x|)^{\alpha-\kappa}\, dx\notag\\
&\hspace{-10em}\lesssim2^{-\nu (L+\alpha_1-\frac n{\inf u})} \norm{\lambda}{\mubwpinfx} \mathfrak{p}_{\max\{\lceil \kappa \rceil,L\}}(\phi) ,\label{eq:molconv2}
\end{align}
where in the last line we have chosen $\kappa$ such that $\kappa>\alpha+n$ and $M>L+\kappa+\alpha+n$ and have used the estimate
\begin{align}\label{sum_on_m}
\sum_{m\in\Z^n}(1+2^\nu|x-2^{-\nu}m|)^{-M+L+\kappa+\alpha}&=\sum_{m\in\Z^n}(1+|2^\nu x-m|)^{-M+L+\kappa+\alpha} \notag \\
&\lesssim\sum_{m'\in\Z^n}(1+|m'|)^{-M+L+\kappa+\alpha}<\infty.
\end{align}

From \eqref{eq:molconv2} in the case $\nu \in \N$ and \eqref{eq:4.4*1/2} in the case $\nu=0$ the conclusion \eqref{eq:4.4*} follows easily, due to our hypothesis on $L$.

\smallskip

\underline{Forth step:}
The proof of the last statement of the theorem follows similarly as in \cite[7.1, Step 3]{AC16b}.
\end{proof}
An easy consequence of the considerations above is the following embedding.
\begin{cor}\label{cor:EinSS}
Let $\vek{w}\in\mgk$, $p, q \in \Plog$ and $u \in\mathcal{P}(\Rn)$ with $0<p^-\leq p(x)\leq u(x)\leq \sup u < \infty$ and  $q^-,q^+\in(0,\infty)$. Then it holds
\begin{align*}
\Mufwpqx\hookrightarrow\SSn.
\end{align*}
\end{cor}
\begin{proof}
Given any $\phi\in\Sn$ we want to prove that there exists a $c_\phi>0$ such that
\begin{align*}
|\langle f,\phi\rangle|\leq c_\phi\norm{f}{\Mufwpqx}\quad\text{for any }f\in\Mufwpqx.
\end{align*}
Let $K,L,M$ be as in Theorem \ref{convergenceS'} and consider arbitrary $\phi\in\Sn$ and $f\in\Mufwpqx$. Use such $K,L$ (and some $d$) in Theorem \ref{1st_direction} and write, as there and for appropriate coefficients and atoms,
\begin{align}\label{eq:fsum(f)}
f=\sum_{\nu=0}^\infty\sum_{m\in\Z^n}\lambda_{\nu,m}(f)\, a_{\nu,m},\quad\text{convergence in }\SSn
\intertext{with}
\norm{\lambda(f)}{\mufwpqx}\leq c_1\norm{f}{\Mufwpqx},\label{eq:seqtofs}
\end{align}
where $c_1>0$ is independent of $f$. Set $\mu_{\nu,m}:=(1+d\sqrt{n}/2)^{-M}a_{\nu,m}$, thus obtaining $[K,L,M]$-molecules concentrated near $Q_{\nu,m}$. Using these molecules and the above coefficients $\lambda(f)$ in Theorem \ref{convergenceS'}, we get from the arguments in the third step of its proof and from Remark \ref{smallen}(ii) that 
\begin{align}
\label{eq:thirdstepest}
\sum_{\nu=0}^\infty\left|\int_\Rn\sum_{m\in\Z^n}\lambda_{\nu,m}(f)\mu_{\nu,m}(x)\phi(x)dx\right| &\leq c_2\norm{\lambda(f)}{\mubwpinfx} \mathfrak{p}_N(\phi) \nonumber \\
&\leq c_3\norm{\lambda(f)}{\mufwpqx} \mathfrak{p}_N(\phi),
\end{align}
where $c_3>0$ and $N \in \N$ are independent of $\lambda(f)$ and $\phi$. Now, we use the convergence of \eqref{eq:fsum(f)} in $\SSn$ and also the fact that the inner sum converges in $\SSn$ to the corresponding pointwise sum and obtain from \eqref{eq:thirdstepest} and \eqref{eq:seqtofs} that
\begin{align*}
|\langle f,\phi\rangle|&\leq\sum_{\nu=0}^\infty\left|\int_\Rn\sum_{m\in\Z^n}\lambda_{\nu,m}(f)a_{\nu,m}(x)\phi(x)dx\right|\\
&=(1+d\sqrt{n}/2)^M\sum_{\nu=0}^\infty\left|\int_\Rn\sum_{m\in\Z^n}\lambda_{\nu,m}(f)\mu_{\nu,m}(x)\phi(x)dx\right|\\
&\leq c_3(1+d\sqrt{n}/2)^M\norm{\lambda(f)}{\mufwpqx} \mathfrak{p}_N(\phi)\\
&\leq c_1c_3 (1+d\sqrt{n}/2)^M\mathfrak{p}_N(\phi)\norm{\lambda(f)}{\Mufwpqx},
\end{align*}
which gives the desired estimate with $c_\phi=c_1c_3 (1+d\sqrt{n}/2)^M\mathfrak{p}_N(\phi)$.
\end{proof}

Now we can prove the completeness of $\Mufwpqx$ by using standard arguments. 

\begin{cor} \label{cor:completeness}
The spaces $\Mufwpqx$ according to Definition \ref{def:TLM} are complete.
\end{cor}
\begin{proof}
Let $(f_m)_{m\in\N}$ be a Cauchy sequence in $\Mufwpqx$. By the previous corollary and the completeness of $\SSn$, there exists $f\in\SSn$ such that $\lim_{m \to \infty}f_m = f$ in $\SSn$.

Given any $\varepsilon>0$, let $m_0\in\N$ be such that, for $l,m\geq m_0$,
\begin{equation}
\sup_{x\in\Rn,r>0}r^{n\left(\frac1{u(x)}-\frac1{p(x)}\right)}\norm{\left(\sum_{j=0}^\infty|w_j(\varphi_j\widehat{f_l-f_m})^\vee|^\q\right)^{1/\q}\!\!\!\!\!\! \chi_{B_r(x)}}{L_{p(\cdot)}(\Rn)} < \varepsilon.
\label{eq:Cauchy}
\end{equation}

Clearly, given $m\geq m_0$, $x\in\Rn$, $r>0$ and $J\in \N$, we have, pointwisely,
$$\left(\sum_{j=0}^J|w_j(\varphi_j\widehat{f_l-f_m})^\vee|^\q\right)^{1/\q}\!\!\!\!\!\! \chi_{B_r(x)} \;\; {\mathop {\longrightarrow}_{l \rightarrow \infty}\;\;} \left(\sum_{j=0}^J|w_j(\varphi_j\widehat{f-f_m})^\vee|^\q\right)^{1/\q}\!\!\!\!\!\! \chi_{B_r(x)}.$$
On the other hand, from \eqref{eq:Cauchy} and the lattice property of $\Lp$, for $l\geq m_0$ the $\Lp$-quasi-norm of the left-hand side above is bounded above by $r^{n\left(\frac1{p(x)}-\frac1{u(x)}\right)} \varepsilon$. Since $\Lp$ satisfies Fatou's lemma (see \cite[after Lemma 3.2.10]{DHHR} for the case $p^-\geq 1$ and play with the property \eqref{eq:Lp/t} to extend it to all values of $p$), then also
$$\norm{\left(\sum_{j=0}^J|w_j(\varphi_j\widehat{f-f_m})^\vee|^\q\right)^{1/\q}\!\!\!\!\!\! \chi_{B_r(x)}}{\Lp} \leq r^{n\left(\frac1{p(x)}-\frac1{u(x)}\right)} \varepsilon.$$
Applying again the just mentioned Fatou's lemma, but now considering $J\to \infty$, we get the above inequality with $J$ replaced by $\infty$, and finally, multiplying both members by $r^{n\left(\frac1{u(x)}-\frac1{p(x)}\right)}$ and applying the supremum for all $x \in \Rn$ and $r>0$, we get $\norm{f-f_m}{\Mufwpqx} \leq \varepsilon$.

So $f=(f-f_m)+f_m \in \Mufwpqx$ and $\lim_{m \to \infty}f_m = f$ in $\Mufwpqx$.

\end{proof}
In some cases it is possible to weaken the hypothesis on $L$ in Theorem \ref{convergenceS'}, at the expense of strengthening the hypothesis on $M$. As a preparatory result, we first prove some kind of Sobolev embedding for the sequence spaces.

\begin{lem} \label{lem:sobolev}
Let $\vek{w}^0\in\mgk$, $p_0, p_1 \in \Plog$ with $p_0 \leq p_1$ and $q, u_0, u_1 \in\mathcal{P}(\Rn)$ with $p_0 \leq u_0$ and
\begin{equation}
\frac{1}{u_0(x)}-\frac{1}{p_0(x)} = \frac{1}{u_1(x)}-\frac{1}{p_1(x)}, \quad x \in \Rn.
\label{eq:sobolev1}
\end{equation}
Let $\vek{w}^1$ be defined by
\begin{equation}
w^1_\nu(x)=w^0_\nu(x) 2^{-\nu n\left(\frac{1}{p_0(x)}-\frac{1}{p_1(x)}\right)}, \quad x \in \Rn, \quad \nu \in \N_0.
\label{eq:sobolev2}
\end{equation}
Then
\begin{equation}
\mathit{n}^{\vek{w}^0,u_0(\cdot)}_{p_0(\cdot),\infty} \hookrightarrow \mathit{n}^{\vek{w}^1,u_1(\cdot)}_{p_1(\cdot),\infty}.
\label{eq:sobolev3}
\end{equation}

\end{lem}

\begin{proof}
By \eqref{eq:sobolev1} and the hypothesis $p_0\leq u_0$, it is clear that also $p_1 \leq u_1$. On the other hand, from the fact that $\vek{w}^0$ is an admissible weight sequence and the definition \eqref{eq:sobolev2} it is not difficult to see that $\vek{w}^1$ is also an admissible weight sequence (possibly for different parameters). So, both spaces in \eqref{eq:sobolev3} are well defined.

Observe that, by Lemmas \ref{lem:lem3.7} and \ref{lem:lem3.9} and hypothesis \eqref{eq:sobolev1}, we have that
\begin{equation}
\norm{\lambda}{\mathit{n}^{\vek{w^1},u_1(\cdot)}_{p_1(\cdot),\infty}} \approx \sup_{\nu\in\N_0} \sup_{j(\leq \nu),x,k} 2^{-jn\left(\frac{1}{u_0(x)}-\frac{1}{p_0(x)}\right)} \norm{\sum_{m\in\Zn} w^1_\nu(2^{-\nu}m) \lambda_\num \chi_\num \chi_{Q_{j,k}}}{L_{p_1(\cdot)}(\Rn)},
\label{eq:reduction}
\end{equation}
where the inner supremum runs over all $j\in\Z$ with $j\leq \nu$, $x\in \Rn$ and $k\in \Zn$ with $|x-2^{-j}k|_\infty\leq \frac{3}{2}2^{-j}$. Given $\nu\in\N_0$ and such $j$, $x$ and $k$, by the dyadic structure we have either $Q_\num \subset Q_{j,k}$ or $Q_\num \cap \overset{\circ}{Q}_{j,k} = \emptyset$. Define
$$\gamma_{\mu,m}^{j,k} := \left\{  \begin{array}{ll}
	\lambda_{\mu,m} & \mbox{ if } Q_{\mu,m} \subset Q_{j,k} \\[2mm]
	0 & \mbox{ otherwise },
	\end{array} \right., \quad \mu\in\N_0,\; m\in\Zn,$$
and notice that
$$\gamma_{\mu,m}^{j,k} \chi_{\mu,m} = \lambda_{\mu,m} \chi_{\mu,m} \chi_{\overset{\circ}{Q}_{j,k}} \; \mbox{ for } \;\mu \geq j.$$
We have then, also with the help of \cite[Prop. 3.9]{Kempka10}, that
\begin{eqnarray*}
\lefteqn{\norm{\sum_{m\in\Zn} w^1_\nu(2^{-\nu}m) \lambda_\num \chi_\num \chi_{Q_{j,k}}}{L_{p_1(\cdot)}(\Rn)}} \\
& \leq & \sup_{\mu\in\N_0} \norm{\sum_{m\in\Zn} w^1_\mu(2^{-\mu}m) \gamma_{\mu,m}^{j,k} \chi_{\mu,m}}{L_{p_1(\cdot)}(\Rn)} 
 =  \norm{\gamma^{j,k}}{\mathit{b}^{\vek{w^1}}_{p_1(\cdot),\infty}}\\
& \lesssim&  \norm{\gamma^{j,k}}{\mathit{b}^{\vek{w^0}}_{p_0(\cdot),\infty}}  =  \sup_{\mu \geq j} \norm{\sum_{m\in\Zn} w^0_\mu(2^{-\mu}m) \gamma_{\mu,m}^{j,k} \chi_{\mu,m}}{L_{p_0(\cdot)}(\Rn)} \\
& \leq & \sup_{\mu\in\N_0} \norm{\sum_{m\in\Zn} w^0_\mu(2^{-\mu}m) \lambda_{\mu,m} \chi_{\mu,m}}{L_{p_0(\cdot)}(Q_{j,k})}.
\end{eqnarray*}
Resuming from \eqref{eq:reduction} and taking the above and Lemma \ref{lem:lem3.7} into consideration we finally get, with $\sup_{j,x,k}$ meaning supremum running over all $j\in\Z$, $x\in \Rn$ and $k\in \Zn$ with $|x-2^{-j}k|_\infty\leq \frac{3}{2}2^{-j}$,
\begin{eqnarray*}
\norm{\lambda}{\mathit{n}^{\vek{w^1},u_1(\cdot)}_{p_1(\cdot),\infty}} & \lesssim & \sup_{j,x,k} 2^{-jn\left(\frac{1}{u_0(x)}-\frac{1}{p_0(x)}\right)} \sup_{\mu\in\N_0} \norm{\sum_{m\in\Zn} w^0_\mu(2^{-\mu}m) \lambda_{\mu,m} \chi_{\mu,m}}{L_{p_0(\cdot)}(Q_{j,k})} \\
& = & \sup_{\mu\in\N_0} \sup_{j,x,k} 2^{-jn\left(\frac{1}{u_0(x)}-\frac{1}{p_0(x)}\right)} \norm{\sum_{m\in\Zn} w^0_\mu(2^{-\mu}m) \lambda_{\mu,m} \chi_{\mu,m}}{L_{p_0(\cdot)}(Q_{j,k})} \\
& \approx & \norm{\lambda}{\mathit{n}^{\vek{w^0},u_0(\cdot)}_{p_0(\cdot),\infty}}.
\end{eqnarray*}
The interchange of the suprema above is possible, since they are taken with respect to distinct sets of parameters.
\end{proof}

\begin{rem}\label{rem:p0p1u0}
Given $p_0, p_1, u_0$ as in the lemma above, it is not always possible to find $u_1 \in \mathcal{P}(\Rn)$ with \eqref{eq:sobolev1}. That is, there are times when the lemma cannot be used. We can always define an extended real-valued measurable function $u_1$ in $\Rn$ by means of the following equation equivalent to \eqref{eq:sobolev1}:
\begin{equation}
\frac{1}{u_1(x)} = \frac{1}{u_0(x)}-\frac{1}{p_0(x)} +\frac{1}{p_1(x)}, \quad x \in \Rn.
\label{eq:sobolev4}
\end{equation}
\noindent Actually, if there is an $u_1$ satisfying the lemma it should be defined in this way. However, such an $u_1$ satisfies the lemma if and only if the right-hand side of \eqref{eq:sobolev4} is non-negative for every $x \in \Rn$.
\end{rem}
\begin{rem} \label{rem:sobolevA}
Since later on we would like to apply the above lemma when $p_0^- \leq 1$ and $p_1(x)=\frac{p_0(x)}{t}$, for some choice of $t \in (0,p_0^-)$, let us explore a little bit the implications of the necessary and sufficient condition for the existence of $u_1$ according to the lemma, following the considerations from Remark \ref{rem:p0p1u0}. In this particular case we have
$$\frac{1}{u_0(x)}-\frac{1}{p_0(x)} +\frac{t}{p_0(x)} \geq 0, \quad x \in \Rn,$$
\noindent and it is not difficult to see that a choice of $t \in (0,p_0^-)$ is possible if and only if
\begin{equation}
\sup_{x \in \Rn} \left(1-\frac{p_0(x)}{u_0(x)}\right) < p_0^-.
\label{eq:sobolev5}
\end{equation}
Here we interpreted $\frac{\infty}{\infty}$ to be $1$. Then $t$ must necessarily be in $\displaystyle\left[ \sup_{x \in \Rn} \left(1-\frac{p_0(x)}{u_0(x)}\right),p_0^- \right)$ and any $t$ in the interior of such an interval will do.
\end{rem}
\begin{rem} \label{rem:sobolevB}
We would also like to remark that under the mere conditions $p_0 \in \Plog$, $u_0 \in\mathcal{P}(\Rn)$ and $p_0^+<\infty$ one has that \eqref{eq:sobolev5} implies that $\frac{1}{p_0^-}-1 < \frac{1}{\inf u_0}$ when $p_0^- \leq 1$. Note that then necessarily $u_0$ cannot assume the value $\infty$. Further, when $p_0^- > 1$ then condition \eqref{eq:sobolev5} is trivially verified, as well as the condition $0 \leq \frac{1}{\inf u_0}$. Summing up and using the standard notation $\sigma_r := n\left(\frac{1}{\min\{1,r\}}-1\right)$ we have that
\begin{equation}
\sup_{x \in \Rn} \left(1-\frac{p_0(x)}{u_0(x)}\right) < p_0^- \; \Rightarrow \; \sigma_{p_0^-} \leq \frac{n}{\inf u_0},
\label{eq:sobolev6}
\end{equation}
under the conditions $p_0 \in \Plog$, $u_0 \in\mathcal{P}(\Rn)$ and $p_0^+<\infty$. When we further assume that $u_0$ is not identically equal to $\infty$, then \eqref{eq:sobolev6} can be written with both inequalities strict.
\end{rem}

For the next lemma see \cite[Lemma 7.1]{AC16b} and references therein.

\begin{lem}\label{lem:hnum}
Let $j,\nu\in\N_0$, $x\in\Rn$, $0<t\leq1$ and $R>n/t$. Then for all $(h_\num)_m\subset\mathbb{C}$
\begin{align*}
&\sum_{m\in\Z^n}|h_\num|(1+2^{\min(\nu,j)}|x-2^{-\nu}m|)^{-R}\\
&\hspace{5em}\lesssim\max(1,2^{(\nu-j)R})\left(\eta_{\nu,Rt}\ast\left|\sum_{m\in\Z^n}h_\num\chi_\num\right|^t\right)^{1/t}(x),
\end{align*}
where the involved constant is independent of $\nu,j,x$ and $(h_\num)_m$.
\end{lem}

We can now state and prove an alternative version of Theorem \ref{convergenceS'} where the hypothesis imposed on $L$ is weaker. In order to make things more readable we introduce the abbreviation 
\begin{equation*}
c_\infty(1/p,1/u,t):=\max\left(0,\sup_{x\in\Rn}\left(\frac1{p(x)}-\frac1{u(x)}\right)-\frac{t}{p_\infty}\right).
\label{eq:cput}
\end{equation*}
So, in particular we have $c_\infty(1/p,1/u,1)=\cpu$ from \eqref{eq:cpu}.

\begin{thm} \label{convergenceS'B}
Let $\vek{w}\in\mgk$, $p \in \Plog$ and $q, u \in\mathcal{P}(\Rn)$ with $0<p^-\leq p(x)\leq u(x)\leq \sup u < \infty$ and $\sup_{x \in \Rn} \left(1-\frac{p(x)}{u(x)}\right) < p^-$. Let $\lambda\in\mubwpinfx$ and $(\mu_{\nu,m})_\num$ be $[K,L,M]$-molecules with 
\begin{align*}
L&>-\alpha_1+\sigma_{p^-}
\intertext{and}
M&>L+2n+2\alpha+2p^-c_{\log}(1/p)\sigma_{p^-}+nc_\infty(1/p,1/u,\min\{1,p^-\}).
\end{align*}
Then 
\begin{align}\label{eq:fsumB}
\sum_{\nu=0}^\infty\sum_{m\in\Z^n}\lambda_{\nu,m}\mu_{\nu,m}\quad\text{converges in }\SSn
\end{align}
and the convergence in $\SSn$ of the inner sum gives the regular distribution obtained by taking the corresponding pointwise convergence. Moreover, the sum 
\begin{align*}
\sum_{(\nu,m)\in\N_0\times\Z^n}\lambda_{\nu,m}\mu_{\nu,m}\quad\text{converges also in }\SSn
\end{align*}
to the same distribution as the iterated sum in \eqref{eq:fsumB}.
\end{thm}

\begin{proof}
\underline{First step:} Clearly, the first and second steps of the proof of Theorem \ref{convergenceS'} also work here, so we have as well that the inner sum in \eqref{eq:fsumB} converges both pointwisely a.e. and in $\SSn$. In particular, \eqref{eq:4.4*1/2} holds under the conditions stated there.

\smallskip

\underline{Second step:} Here we assume that $p^->1$ and show the convergence in $\SSn$ of the outer sum in \eqref{eq:fsumB}. As in the third step of the proof of Theorem \ref{convergenceS'}, the mentioned convergence follows if we show that there exists $N\in\N$ and $c>0$ such that \eqref{eq:4.4*} holds for all $\phi \in \Sn$.

Consider $\nu \in \N$. We proceed similarly as in the third step of the proof of Theorem \ref{convergenceS'} up to the point where \eqref{eq:weightestimate} and \eqref{eq:easyestimate} are used, therefore getting
\begin{align}
\left|\int_\Rn\sum_{m\in\Z^n}\lambda_\num\mu_\num(x)\phi(x)dx\right|& \notag \\
&\hspace{-11em}\lesssim2^{-\nu (L+\alpha_1)} \mathfrak{p}_{\max\{\lceil \kappa \rceil,L\}}(\phi) \notag \\ 
&\hspace{-9em} \times \int_\Rn \sum_{m\in\Z^n} |\lambda_\num| w_\nu(2^{-\nu}m) (1+2^\nu|x-2^{-\nu}m|)^{-M+L+\kappa+\alpha}(1+|x|)^{\alpha-\kappa}\, dx. \label{eq:molconv2B}
\end{align}

\noindent Now we choose $j_0\in\N_0$ in such a way that $\sqrt{n}2^{-j_0-1}\leq 1$ and choose $\kappa \in (n+\alpha,M-L-n(1+\cpu)-\alpha)$ and estimate the integral above by
\begin{align*}
\sum_{k\in \Z^n} \int_{Q_{j_0,k}} (1+|x|)^{\alpha-\kappa} \sum_{m\in\Z^n} |\lambda_\num| w_\nu(2^{-\nu}m) (1+2^\nu|x-2^{-\nu}m|)^{-M+L+\kappa+\alpha}dx&\\
&\hspace{-30em} \lesssim \sum_{k\in \Z^n} (1+|k|)^{\alpha-\kappa} \int_{Q_{j_0,k}} \left( \eta_{\nu,R} \ast \sum_{m\in\Z^n} |\lambda_\num| w_\nu(2^{-\nu}m) \chi_{\nu,m} \right)(x) \, dx,
\end{align*}
where here we have used Lemma \ref{lem:hnum} with $j=\nu$, $t=1$ and $R=M-L-\kappa-\alpha$. Notice that the choice of $\kappa$ guarantees that $M-L-\kappa-\alpha>n$. We proceed by applying H\"older's inequality in the integral, Lemma \ref{lem:xinorm} and afterwards the scalar case of Theorem \ref{thm:MorreyHardy} 
and using the fact that our choice of $\kappa$ guarantees that $R>n(1+\cpu)$. So, we estimate further
\begin{align*}
& \lesssim \sum_{k\in \Z^n} (1+|k|)^{\alpha-\kappa} \norm{\left( \eta_{\nu,R} \ast \sum_{m\in\Z^n} |\lambda_\num| w_\nu(2^{-\nu}m) \chi_{\nu,m} \right) \chi_{B_1(2^{-j_0}k)}}{L_{\p}(\Rn)} \\
&\hspace{2cm} \times \norm{\chi_{B_1(2^{-j_0}k)}}{L_{p'(\cdot)}(\Rn)} \\
& \lesssim \left( \sum_{k\in \Z^n} (1+|k|)^{\alpha-\kappa} \right) \norm{ \eta_{\nu,R} \ast \sum_{m\in\Z^n} |\lambda_\num| w_\nu(2^{-\nu}m) \chi_{\nu,m}}{\Mup} \\
& \lesssim \norm{\sum_{m\in\Z^n} |\lambda_\num| w_\nu(2^{-\nu}m) \chi_{\nu,m}}{\Mup} \\
& \leq \norm{\lambda}{\mubwpinfx}.
\end{align*}
Inserting this estimate in \eqref{eq:molconv2B} we get
\begin{equation*}
\left|\int_\Rn\sum_{m\in\Z^n}\lambda_\num\mu_\num(x)\phi(x)dx\right| \lesssim 2^{-\nu (L+\alpha_1)} \norm{\lambda}{\mubwpinfx} \mathfrak{p}_{\max\{\lceil \kappa \rceil,L\}}(\phi). 
\end{equation*}

From this in the case $\nu \in \N$ and \eqref{eq:4.4*1/2} in the case $\nu=0$ the conclusion \eqref{eq:4.4*} follows easily due to our hypothesis on $L$.

\smallskip

\underline{Third step:} Now we assume that $p^- \in (0,1]$ and show the convergence in $\SSn$ of the outer sum in \eqref{eq:fsumB}. As in the previous step, the mentioned convergence follows if we show that there exist $N\in\N$ and $c>0$ such that \eqref{eq:4.4*} holds for all $\phi \in \Sn$.

Let $p_0:=p$, $u_0:=u$, $\vek{w}^0:=\vek{w}$ and consider $t \in \left(\sup_{x \in \Rn} \left(1-\frac{p_0(x)}{u_0(x)}\right), p_0^-\right)$, $p_1(\cdot):=\frac{p_0(\cdot)}{t}$ and $w^1_\nu(\cdot):=w^0_\nu(\cdot) 2^{-\nu n\left(\frac{1}{p_0(\cdot)}-\frac{1}{p_1(\cdot)}\right)} = w_\nu(\cdot) 2^{-\nu n\frac{1-t}{p(\cdot)}}$, $\nu \in \N_0$. Such a choice of $t$ is possible due to our hypothesis $\sup_{x \in \Rn} \left(1-\frac{p(x)}{u(x)}\right) < p^-$. Since $p_0 \in \Plog$, also $p_1 \in \Plog$. On the other hand, from $\vek{w}^0 \in \mgk$ it follows, with the help of \cite[Example 2.5]{AC16}, that $\vek{w}^1 \in \mathcal{W}^\beta_{\beta_1,\beta_2}$ with $\beta=\alpha+n(1-t)c_{\log}(1/p)$, $\beta_1=\alpha_1-n(1-t)/p^-$ and $\beta_2=\alpha_2-n(1-t)/p^+$. Consider also $u_1 \in\mathcal{P}(\Rn)$ given by $\frac{1}{u_1(\cdot)} = \frac{1}{u_0(\cdot)}-\frac{1}{p_0(\cdot)} +\frac{1}{p_1(\cdot)}$. Our hypotheses and the discussion in Remark \ref{rem:sobolevA} guarantee that this is possible. We have, moreover, that $p_0(x) \leq p_1(x) \leq u_1(x) \leq \sup u_1 < \infty$, so that applying Lemma \ref{lem:sobolev} to $\lambda \in \mubwpinfx = \mathit{n}^{\vek{w}^0,u_0(\cdot)}_{p_0(\cdot),\infty}$ we get that also $\lambda \in \mathit{n}^{\vek{w}^1,u_1(\cdot)}_{p_1(\cdot),\infty}$. Since also $p_1^->1$, we can thus apply the second step above to such a $\lambda$ and for the parameters $\vek{w}^1$, $p_1$ and $u_1$ as long as
$$L>-\alpha_1+n(1-t)/p^-$$
and
$$M>L+2n+2\alpha+2n(1-t)c_{\log}(1/p)+nc_\infty(1/p,1/u,t).$$
Due to our hypotheses on $L$ and $M$ in the statement of the theorem it is indeed possible to choose the $t$ above so that also the two last inequalities are fulfilled. Therefore our desired conclusion \eqref{eq:4.4*} also holds for the case considered in this step.

\smallskip

\underline{Forth step:}
The proof of the last statement of the theorem follows similarly as in \cite[7.1, Step 3]{AC16b}.
\end{proof}

\begin{rem}
Comparing Theorems \ref{convergenceS'} and \ref{convergenceS'B}, we see that they provide the same conclusions under different hypotheses. In Theorem \ref{convergenceS'B} $\sup u<\infty$ is further assumed and the requirement for $M$ is in general stronger. On the other hand, according to the discussion in Remark \ref{rem:sobolevB}, the requirement for $L$ is weaker, though in the case $p^-<1$ it only applies under the extra condition $\sup_{x \in \Rn} \left(1-\frac{p(x)}{u(x)}\right) < p^-$. 
We would also like to remark that when $\sup_{x \in \Rn} \left(1-\frac{p(x)}{u(x)}\right) \geq p^-$ then it follows by straightforward calculations that $\sigma_{p^-} \geq \frac{n}{\sup u}$. Though this does not guarantee that then the requirement for $L$ in Theorem \ref{convergenceS'B} is stronger than in Theorem \ref{convergenceS'}, in the case when $u$ happens to be constant it definitely shows that it is not weaker.
\end{rem}

\begin{rem} \label{improvement}
It is possible to improve slightly the conditions on $p$ and $u$ in the above theorem, by not imposing its boundedness from above. The boundedness was used in the proof only to apply the scalar case of Theorem \ref{thm:MorreyHardy} and the scalar case holds without such restriction, as can be seen in \cite{AC18}. 
\end{rem}

Before coming to the theorem giving conditions for a linear combination of molecules to be in our functions spaces and showing the opposite direction of Theorem \ref{1st_direction}, we still need some preparatory results. The next one estimates convolutions of molecules with functions constructed from admissible pairs. See \cite[Lemma 7.2]{AC16b} and references therein for its proof.

\begin{lem}\label{lem:muconv}
Let $(\varphi_j)_{j\in\N_0}$ be a system constructed from any given admissible pair according to Definition \ref{dfn:admissiblePair} and $(\mu_\num)$ be $[K,L,M]$-molecules. Then, for $M>L+n$ and $N \in [0,M-L-n]$,
\begin{align*}
|(\varphi_j^\vee\ast\mu_\num)(x)|&\lesssim2^{-(\nu-j)(L+n)}(1+2^j|x-2^{-\nu}m|)^{-N}\quad\text{for }j\leq\nu
\intertext{and}
|(\varphi_j^\vee\ast\mu_\num)(x)|&\lesssim2^{-(j-\nu)K}(1+2^\nu|x-2^{-\nu}m|)^{-M}\quad\text{for }j\geq\nu,
\end{align*}
with the implicit constants independent of $x\in\Rn$, $m\in\Z^n$, $j,\nu\in\N_0$ and, as long as $K,L,M$ are kept fixed, of the particular system of molecules taken.
\end{lem}

\begin{thm} \label{conv_convolution}
Under the hypotheses of Theorem \ref{convergenceS'}, we have also that
\begin{equation}
\sum_{\nu=0}^\infty \sum_{m\in\Z^n} \lambda_{\nu,m}(\varphi_j^\vee\ast\mu_\num)
\label{eq:sum_convolution}
\end{equation}
converges both in $\SSn$ and pointwisely a.e. to the same distribution, where $(\varphi_j)_{j\in\N_0}$ is a system constructed from any given admissible pair according to Definition \ref{dfn:admissiblePair}.
\end{thm}

\begin{proof}
The convergence in $\SSn$ is clear from what we have already obtained in Theorem \ref{convergenceS'}.

That the inner sum converges pointwisely a.e. to the corresponding limit in $\SSn$ can be proved similarly as in the second step of the proof of Theorem \ref{convergenceS'}, now using the estimates in Lemma \ref{lem:muconv} instead of properties of molecules. However, in the case when $\nu>j$, instead of \eqref{eq:weightestimate} one should use the estimate
\begin{align}\label{eq:weightestimate2}
1\lesssim 2^{-\nu\alpha_1}w_\nu(2^{-\nu}m)(1+|x|)^\alpha(1+2^j|x-2^{-\nu}m|)^\alpha,
\end{align} 
which is also an easy consequence of the properties of a weight sequence, and afterwards the estimate
\begin{align*}
\sum_{m\in\Z^n}(1+2^j|x-2^{-\nu}m|)^{\alpha-N}& \leq c\, 2^{(\nu-j)n}, \quad \mbox{choosing }\, N>\alpha +n,
\end{align*}
with $c>0$ independent of $\nu, j \in \N_0$ (with $\nu \geq j$) and $x \in \Rn$, which is a direct consequence of \cite[Lemma 3.7]{Kempka10}. 

So, in the case $\nu>j$ one gets, in this way, and for any integer $\kappa > n + \alpha$,
\begin{equation}
\int_\Rn \sum_{m\in\Z^n} |\lambda_{\nu,m}(\varphi_j^\vee\ast\mu_\num)(x)\phi(x)|\, dx \leq c\, 2^{-\nu (L+\alpha_1-\frac n{\inf u})} 2^{jL} \norm{\lambda}{\mubwpinfx} \mathfrak{p}_{\kappa}(\phi),
\label{eq:est_sum_m}
\end{equation}
where $c>0$ is independent of $\nu,j \in \N_0$ (with $\nu \geq j$), $\lambda \in \mubwpinfx$ and $\phi \in \Sn$.

In order to prove that then the outer sum in \eqref{eq:sum_convolution} converges also pointwisely a.e. to the corresponding limit in $\SSn$ it suffices clearly to deal with the sum for $\nu>j$, as the remaining sum is finite. The idea is again to start as in the second step of the proof of Theorem \ref{convergenceS'}. We obtain with the help of \eqref{eq:est_sum_m} and the hypothesis on $L$
\begin{eqnarray} \label{eq:est_sum_nu>j}
\lefteqn{\int_\Rn \sum_{\nu=j+1}^\infty \left| \sum_{m\in\Z^n} \lambda_{\nu,m}(\varphi_j^\vee\ast\mu_\num)(x)\phi(x) \right| \, dx} \\
& \lesssim & 2^{jL} \sum_{\nu=j+1}^\infty 2^{-\nu (L+\alpha_1-\frac n{\inf u})} \norm{\lambda}{\mubwpinfx} \mathfrak{p}_{\kappa}(\phi) \nonumber \\
& \lesssim & 2^{-j(\alpha_1-\frac n{\inf u})} \norm{\lambda}{\mubwpinfx} \mathfrak{p}_{\kappa}(\phi), \nonumber
\end{eqnarray}
with the implicit constant independent of $j \in \N_0$, $\lambda \in \mubwpinfx$ and $\phi \in \Sn$.
\end{proof}

\begin{thm} \label{conv_convolutionB}
Under the hypotheses of Theorem \ref{convergenceS'B}, except that for $L$ we impose the stronger condition
$$L > -\alpha_1+\sigma_{p^-} + n c_\infty(1/p,1/u,\min\{1,p^-\}),$$
we have that
\begin{equation}
\sum_{\nu=0}^\infty \sum_{m\in\Z^n} \lambda_{\nu,m}(\varphi_j^\vee\ast\mu_\num)
\label{eq:sum_convolutionB}
\end{equation}
converges both in $\SSn$ and pointwisely a.e. to the same distribution, where $(\varphi_j)_{j\in\N_0}$ is a system constructed from any given admissible pair according to Definition \ref{dfn:admissiblePair}.
\end{thm}

\begin{proof}
The convergence in $\SSn$ is clear from what we have already obtained in Theorem \ref{convergenceS'B}.

The proof that the inner sum converges pointwisely a.e. to the corresponding limit in $\SSn$ can be done exactly as in Theorem \ref{conv_convolution}.

To prove that the outer sum in \eqref{eq:sum_convolutionB} converges also pointwisely a.e. to the corresponding limit in $\SSn$, again it suffices to deal with the sum for $\nu>j$, as the remaining sum is finite. And, as in the proof of Theorem \ref{conv_convolution}, to that effect it is enough to get a result like \eqref{eq:est_sum_nu>j}, namely the existence of $J\in \N$ and $c>0$ such that
\begin{equation*} \label{eq:est_sum_nu>jB}
\int_\Rn \sum_{\nu=j+1}^\infty \left| \sum_{m\in\Z^n} \lambda_{\nu,m}(\varphi_j^\vee\ast\mu_\num)(x)\phi(x) \right| \, dx \leq c\, \mathfrak{p}_J(\phi) 
\end{equation*}
for all $\phi \in \Sn$.

We consider first the case $p^->1$. Since 
\begin{eqnarray}
\lefteqn{\int_\Rn \sum_{\nu=j+1}^\infty \left| \sum_{m\in\Z^n} \lambda_{\nu,m}(\varphi_j^\vee\ast\mu_\num)(x)\phi(x) \right| \, dx} \nonumber \\
& \leq & \sum_{\nu=j+1}^\infty \int_\Rn \sum_{m\in\Z^n} \left| \lambda_{\nu,m}(\varphi_j^\vee\ast\mu_\num)(x)\phi(x) \right| \, dx,
\label{eq:triangular}
\end{eqnarray}
we start by estimating the latter integral. Here the situation is somewhat similar to when we estimated \eqref{eq:molconv2B} in the second step of the proof of Theorem \ref{convergenceS'B}, though now we are using Lemma \ref{lem:muconv} and the estimate \eqref{eq:weightestimate2}, so that we get, with $N \in [0,M-L-n]$ and $\kappa >0$ at our disposal,
\begin{align}
\int_\Rn \sum_{m\in\Z^n} \left| \lambda_{\nu,m}(\varphi_j^\vee\ast\mu_\num)(x)\phi(x) \right| \, dx & \notag \\
&\hspace{-11em}\lesssim 2^{-j\alpha_1} 2^{-(\nu-j) (L+n+\alpha_1)} \mathfrak{p}_{\lceil \alpha+\kappa \rceil}(\phi) \notag \\ 
&\hspace{-9em} \times \int_\Rn \sum_{m\in\Z^n} |\lambda_\num| w_\nu(2^{-\nu}m) (1+2^j|x-2^{-\nu}m|)^{\alpha-N}(1+|x|)^{-\kappa}\, dx. \label{eq:convolconv}
\end{align}
Now we deal with the latter integral arguing as after \eqref{eq:molconv2B} in the second step of the proof of Theorem \ref{convergenceS'B}. With $R=N-\alpha$ and choosing $\kappa > n$ as well as $N\in(\alpha+n(1+\cpu),\min\{M-L-n,L+\alpha_1+n+\alpha\})$ we obtain
\begin{align*}
&\int_\Rn \sum_{m\in\Z^n} |\lambda_\num| w_\nu(2^{-\nu}m) (1+2^j|x-2^{-\nu}m|)^{\alpha-N}(1+|x|)^{-\kappa}\, dx \\
&\hspace{10em}\lesssim  2^{(\nu-j)(N-\alpha)} \norm{\lambda}{\mubwpinfx}.
\end{align*}

%
%
Inserting this estimate in \eqref{eq:convolconv} and the latter in \eqref{eq:triangular}, we finally get
\begin{eqnarray*}
\lefteqn{\int_\Rn \sum_{\nu=j+1}^\infty \left| \sum_{m\in\Z^n} \lambda_{\nu,m}(\varphi_j^\vee\ast\mu_\num)(x)\phi(x) \right| \, dx} \\
& \lesssim & \sum_{\nu=j+1}^\infty 2^{-j\alpha_1} 2^{-(\nu-j) (L+n+\alpha_1)} \mathfrak{p}_{\lceil \alpha+\kappa \rceil}(\phi) 2^{(\nu-j)(N-\alpha)} \norm{\lambda}{\mubwpinfx} \\
& \lesssim & 2^{-j\alpha_1} \norm{\lambda}{\mubwpinfx} \mathfrak{p}_{\lceil \alpha+\kappa \rceil}(\phi).
\end{eqnarray*}

As to the case $p^- \in (0,1]$, it can be proved by reduction to the preceding case with the same choices of $p_0$, $u_0$, $\vek{w}^0$, $t$, $p_1$, $u_1$ and $\vek{w}^1$ as in the third step of the proof of Theorem \ref{convergenceS'B}, this being possible by the hypothesis on $M$ and the (new) hypothesis on $L$.
\end{proof}

\begin{rem}
By similar reasons as in Remark \ref{improvement}, it is possible to prove that the above theorem holds without imposing $\sup u < \infty$.
\end{rem}

Finally, we come to the result in the direction opposite of Theorem \ref{1st_direction}.

\begin{thm}\label{thm:416}
Let $\vek{w}\in\mgk$, $p,q\in\Plog$ and $u\in\mathcal{P}(\Rn)$ with $0<p^- \leq p(x)\leq u(x) \leq \sup u < \infty$ and $q^-,q^+\in(0,\infty)$. Let $K,L\in\N_0$ and $M>0$ be such that $K>\alpha_2$ and either
\begin{equation}
L>-\alpha_1+\max\left\{\frac n{\inf u},\sigma_{p^-,q^-}+n\cpu\right\}
\label{eq:L1}
\end{equation}
and
\begin{equation}
M>L+2n+2\alpha+\sigma_{p^-,q^-}+n\cpu
\label{eq:M1}
\end{equation}
or, when $\sup_{x\in\Rn}\left(1-\frac{p(x)}{u(x)}\right)<p^-$ (which necessarily holds when $p^-\geq 1$),
\begin{equation}
L>-\alpha_1+\sigma_{p^-,q^-}+nc_\infty(1/p,1/u,\min\{1,p^-\})
\label{eq:L2}
\end{equation}
and
\begin{equation}
M>L+2n+2\alpha+\max\left\{1,2c_{\log}(1/p)\right\}\sigma_{p^-,q^-}+nc_\infty(1/p,1/u,\min\{1,p^-\}).
\label{eq:M2}
\end{equation}
Let $\lambda\in\mufwpqx$ and $(\mu_{\nu,m})_{\nu,m}$ be $[K,L,M]$-molecules. Then
\begin{align}
f:=\sum_{\nu=0}^\infty\sum_{m\in\Z^n}\lambda_\num\mu_\num, \quad \mbox{convergence in }\, \SSn, \label{eq:f=sum}
\intertext{belongs to $\Mufwpqx$ and there exists a constant $c>0$ independent of $\lambda$ and $(\mu_{\nu,m})_{\nu,m}$ such that}
\norm{f}{\Mufwpqx}\leq c\norm{\lambda}{\mufwpqx}. \label{eq:fnorm}
\end{align}
\end{thm}
\begin{proof}
\underline{First step:} We start by observing that, in addition to the common hypotheses, and together with Remark \ref{smallen}(ii), the set of conditions \eqref{eq:L1}-\eqref{eq:M1} guarantees the application of Theorems \ref{convergenceS'} and \ref{conv_convolution} while the set of conditions \eqref{eq:L2}-\eqref{eq:M2} (together with the extra hypothesis $\sup_{x\in\Rn}\left(1-\frac{p(x)}{u(x)}\right)<p^-$) guarantees the application of Theorems \ref{convergenceS'B} and \ref{conv_convolutionB}. In particular, in both cases we have the convergence in $\SSn$ of the iterated sum in \eqref{eq:f=sum}. We still have to prove that $f$ belongs to $\Mufwpqx$ and that the estimate \eqref{eq:fnorm} holds.

Note that, given a system $(\varphi_j)_{j\in\N_0}$ constructed from any given admissible pair according to Definition \ref{dfn:admissiblePair}, we have, for each $j\in\N_0$ and for a.e. $x\in\Rn$, that
\begin{eqnarray*}
w_j(x)(\varphi_j^\vee\ast f)(x) & = & w_j(x)\left(\sum_{\nu=0}^\infty\sum_{m\in\Z^n}\lambda_\num(\varphi_j^\vee\ast\mu_\num)\right)(x) \\
& = & w_j(x)\sum_{\nu=0}^\infty\sum_{m\in\Z^n}\lambda_\num(\varphi_j^\vee\ast\mu_\num)(x),
\end{eqnarray*}
where the convergence of the sums in the first line is in the sense of $\SSn$ and in the last line is pointwise. The identification between the two being justified by Theorem  \ref{conv_convolution} or by Theorem \ref{conv_convolutionB}, depending on which set of special hypotheses \eqref{eq:L1}-\eqref{eq:M1} or \eqref{eq:L2}-\eqref{eq:M2} (together with $\sup_{x\in\Rn}\left(1-\frac{p(x)}{u(x)}\right)<p^-$) is considered. Whatever the case, we can thus write for each $j\in\N_0$ and for a.e. $x\in\Rn$ that
\begin{eqnarray}
|w_j(x)(\varphi_j^\vee\ast f)(x)| & \leq & \sum_{\nu=0}^j\sum_{m\in\Z^n}|\lambda_\num(\varphi_j^\vee\ast\mu_\num)(x)|w_j(x) \nonumber\\
&  &  + \sum_{\nu=j+1}^\infty\sum_{m\in\Z^n}|\lambda_\num(\varphi_j^\vee\ast\mu_\num)(x)|w_j(x). 
\label{eq:splitting}
\end{eqnarray}


Note also, whatever the special hypotheses chosen we always have 
$$L>-\alpha_1+\sigma_{p^-,q^-}+n\cpu$$
and
$$M>L+2n+\alpha+\sigma_{p^-,q^-}+n\cpu.$$
So, it is possible to choose $t\in (0,\min\{p^-,q^-,1\})$ such that 
$$\alpha+\frac{n}{t}+n\cpu < \min\{ M-L-n,L+n+\alpha_1+\alpha \}$$
and $N$ strictly in between these two numbers. In what follows we assume that $t$ and $N$ have been chosen in such a way.

\smallskip

\underline{Second step:} Here we estimate the two terms which appeared after the splitting of $|w_j(x)(\varphi_j^\vee\ast f)(x)|$ in \eqref{eq:splitting} above. For the first term, we obtain using Lemma \ref{lem:muconv} and the properties of the weight the estimates 
\begin{align}
&\sum_{\nu=0}^j\sum_{m\in\Z^n}|\lambda_\num(\varphi_j^\vee\ast\mu_\num)(x)|w_j(x)\notag\\
&\lesssim\sum_{\nu=0}^j\sum_{m\in\Z^n}2^{-(j-\nu)(K-\alpha_2)}|\lambda_\num|w_\nu(2^{-\nu}m)(1+2^\nu|x-2^{-\nu}m|)^{\alpha-M}\notag\\
&\lesssim\sum_{\nu=0}^j2^{-(j-\nu)(K-\alpha_2)}\left(\eta_{\nu,(M-\alpha)t}\ast\left|\sum_{m\in\Z^n}|\lambda_\num|w_\nu(2^{-\nu}m)\chi_\num\right|^t\right)^{1/t}(x),\label{eq:molecules1}
\end{align}
where we have also used Lemma \ref{lem:hnum} with $R=M-\alpha$. The same reasoning gives for the second term
\begin{align}
&\sum_{\nu=j+1}^\infty\sum_{m\in\Z^n}|\lambda_\num(\varphi_j^\vee\ast\mu_\num)(x)|w_j(x)\notag\\
&\lesssim\sum_{\nu=j+1}^\infty\sum_{m\in\Z^n}2^{-(\nu-j)(L+n+\alpha_1)}|\lambda_\num|w_\nu(2^{-\nu}m)(1+2^j|x-2^{-\nu}m|)^{\alpha-N}\notag\\
&\lesssim\sum_{\nu=j+1}^\infty2^{-(\nu-j)(L+n+\alpha_1-(N-\alpha))}\left(\eta_{\nu,(N-\alpha)t}\ast\left|\sum_{m\in\Z^n}|\lambda_\num|w_\nu(2^{-\nu}m)\chi_\num\right|^t\right)^{1/t}(x),\label{eq:molecules2}
\end{align}
where we have applied again  Lemma \ref{lem:hnum}, now with $R=N-\alpha$. Using the fact that $N-\alpha \leq M-\alpha$ and defining $\delta:=\min\{L+n+\alpha_1-(N-\alpha),K-\alpha_2\}$, we have out of \eqref{eq:molecules1}, \eqref{eq:molecules2} and \eqref{eq:splitting} that, for each $j\in\N_0$ and for a.e. $x\in\Rn$,
\begin{align*}
|w_j(x)(\varphi_j^\vee\ast f)(x)| \lesssim \sum_{\nu=0}^\infty2^{-|\nu-j|\delta}\left(\eta_{\nu,(N-\alpha)t}\ast\left|\sum_{m\in\Z^n}|\lambda_\num|w_\nu(2^{-\nu}m)\chi_\num\right|^t\right)^{1/t}(x).
\end{align*}

\smallskip

\underline{Third step:} Using the last estimate, Lemma \ref{lem:monotony}, the fact that our hypotheses and choices guarantee that $\delta>0$ and $(N-\alpha)t > n+nt\cpu$, Lemma \ref{lem:Hardy}, Lemma \ref{lem:ttrick} and Theorem \ref{thm:MorreyHardy} with $m=(N-\alpha)t$, we finally get
\begin{eqnarray*}
\lefteqn{\norm{f}{\Mufwpqx}\; = \; \norm{\left(|w_j(\varphi^\vee_j\ast f)|\right)_j}{\Muplq}} \\
& \lesssim & \norm{\left(\left(\eta_{\nu,(N-\alpha)t}\ast\left|\sum_{m\in\Z^n}|\lambda_\num|w_\nu(2^{-\nu}m)\chi_\num\right|^t\right)^{1/t}\right)_\nu}{\Muplq}\\
& = & \norm{\left(\eta_{\nu,(N-\alpha)t}\ast\left|\sum_{m\in\Z^n}|\lambda_\num|w_\nu(2^{-\nu}m)\chi_\num\right|^t\right)_\nu}{M_{\p/t}^{\u/t}(\ell_{\q/t})}^{1/t} \\
& \lesssim & \norm{\left(\left|\sum_{m\in\Z^n}|\lambda_\num|w_\nu(2^{-\nu}m)\chi_\num\right|^t\right)_\nu}{M_{\p/t}^{\u/t}(\ell_{\q/t})}^{1/t}\\
& = & \norm{\left(\sum_{m\in\Z^n}|\lambda_\num|w_\nu(2^{-\nu}m)\chi_\num\right)_\nu}{\Muplq}\\
& = & \norm{\left(\sum_{\nu=0}^\infty \sum_{m\in\Z^n}|\lambda_\num w_\nu(2^{-\nu}m)\chi_\num(\cdot)|^\q\right)^{1/\q}}{\Mup} \; = \; \norm{\lambda}{\mufwpqx}.
\end{eqnarray*}
\end{proof}
Using Theorem \ref{thm:416} on the molecular decomposition we also obtain a general embedding on this scale of variable exponent function spaces.
\begin{cor}\label{cor:SinE}
Let $\vek{w}\in\mgk$ be admissible weights and $p,q\in\Plog$, $u\in\mathcal{P}(\Rn)$ with $0<p^-\leq p(x)\leq u(x)\leq \sup u < \infty$ and  $q^-,q^+\in(0,\infty)$. Then it holds
\begin{align*}
\Sn\hookrightarrow\Mufwpqx.
\end{align*}
\end{cor}
\begin{proof}
We have to show that there exist $c>0$ and $N\in\N$ such that
\begin{align*}
\norm{\phi}{\Mufwpqx}\leq c\, \mathfrak{p}_N(\phi)\quad\text{for all }\phi\in\Sn.
\end{align*}
Since for $\phi=0$ there is nothing to prove, consider $0 \not= \phi\in\Sn$. First, observe that given any $K,L\in\N_0$ and $M>0$ we have that $\phi/\mathfrak{p}_{\max(K,\lceil M\rceil)}(\phi)$ is a $[K,L,M]$-molecule concentrated near $Q_{0,0}$. On one hand moment conditions are not required, since $\nu=0$, on the other hand we have for any $0\leq|\beta|\leq K$
\begin{align*}
\left|D^\beta\frac{\phi(x)}{\mathfrak{p}_{\max(K,\lceil M\rceil)}(\phi)}\right|
&=\frac{(1+|x|)^M|D^\beta\phi(x)|}{\mathfrak{p}_{\max(K,\lceil M\rceil)}(\phi)}(1+|x|)^{-M}\\
&\leq 2^{|\beta|0}(1+2^0|x-2^{-0}0|)^{-M}.
\end{align*} 
Observe also that for $\lambda(\phi)$ with $\lambda_{0,0}(\phi):=\mathfrak{p}_{\max(K,\lceil M\rceil)}(\phi)$ and $\lambda_{\nu,m}(\phi):=0$ for $\nu\neq0$ or $m\neq0$ we have that
\begin{align}
\norm{\lambda(\phi)}{\mufwpqx}&=\norm{w_0(0)\mathfrak{p}_{\max(K,\lceil M\rceil)}(\phi)\chi_{0,0}}{\Mup}\notag\\
&=w_0(0)\mathfrak{p}_{\max(K,\lceil M\rceil)}(\phi)\norm{\chi_{0,0}}{\Mup}<\infty\label{eq:last}
\end{align}
by using Lemma \ref{lem:Anorm}. Now fix $K,L,M$ according to Theorem \ref{thm:416} for example with the restrictions \eqref{eq:L1} and \eqref{eq:M1}. Applying this theorem to $\lambda(\phi)$, which has only one non-zero entry, and the molecules above we get 
\begin{align*}
\phi=\lambda_{0,0}(\phi)\frac{\phi}{\mathfrak{p}_{\max(K,\lceil M\rceil)}(\phi)}&\in\Mufwpqx
\intertext{and}
\norm{\phi}{\Mufwpqx}&\leq c_1\norm{\lambda(\phi)}{\mufwpqx},
\intertext{which together with \eqref{eq:last} gives}
\norm{\phi}{\Mufwpqx}&\leq c\, \mathfrak{p}_N(\phi),
\end{align*}
where $c:=c_1w_0(0)\norm{\chi_{0,0}}{\Mup}$ and $N:=\max(K,\lceil M\rceil)$ are independent of $\phi$.
\end{proof}

\end{document}